\documentclass[reqno, 12pt]{amsart}
\usepackage{amsmath}
\usepackage[all]{xy}
\usepackage{amsfonts}
\usepackage{a4wide}
\usepackage{amssymb}
\usepackage{amsthm}

\usepackage{epsfig}
\usepackage{color}

\theoremstyle{plain}
\begingroup
\newtheorem{theorem}{Theorem}[section]

\newtheorem{lemma}[theorem]{Lemma}
\newtheorem{proposition}[theorem]{Proposition}
\newtheorem{corollary}[theorem]{Corollary}

\newtheorem{definition}[theorem]{Definition}
\newtheorem{rmk}[theorem]{Remark}
\newtheorem{example}[theorem]{Example}

\endgroup

\theoremstyle{remark}
\begingroup
\endgroup

\mathchardef\emptyset="001F

\numberwithin{equation}{section}

\newcommand{\op}[1]{{\rm{#1}}}

\newcommand\res{\mathop{\hbox{\vrule height 7pt width .5pt depth 0pt
\vrule height .5pt width 6pt depth 0pt}}\nolimits}

\title[Regularity of $L^p$-vector fields with integer fluxes]
{Interior partial regularity for minimal $L^p$-vector fields with integer fluxes}

\author[Mircea Petrache]{Mircea Petrache}

\begin{document}
\begin{abstract}
 We use a new combinatorial technique to prove the optimal interior partial regularity result for $L^p$-vector fields with integer fluxes that minimize the $L^p$-energy. More precisely, we prove that the minimizing vector fields are H\"{o}lder continuous outside a set that is locally finite inside the domain. The results continue the program started in \cite{PR1}, but this paper is self-contained.
\end{abstract}

 \maketitle
\tableofcontents

\section{Introduction}
\subsection{The result}
In this work we consider vector fields $X\in L^p(B^3, \mathbb R^3)$ such that for every $a\in B^3$ and for almost every $r<\op{dist}(a,\partial B^3)$, 
\begin{equation}\label{star}
\int_{\partial B_r^3(a)} X\cdot \nu \in\mathbb Z, 
\end{equation}
where $\nu:\partial B_r^3(a)\to S^2$ is the outward unit normal vector. We call vector fields satisfying such flux conditions \emph{vector fields with integer fluxes}.\\

We observe that for $p\geq 3/2$ this class reduces to the divergence-free vector fields, and therefore we just look at the ``interesting'' range $p\in[1,3/2[$. It is clear that this class of vector fields is closed under $L^p$-convergence. The compactness result for \emph{weak} convergence holds only for $p>1$ and is the main result of our recent work with Tristan Rivi\`ere \cite{PR1}, and is based upon the introduction of a good distance between slices.\\

Here we concentrate on the interior regularity of $X\in L^p_{\mathbb Z}(B^3, \mathbb R^3)$ that are minima of the $L^p$-energy $||X||_{L^p(B^3)}$. We say that $X$ is a \emph{minimizer} if it achieves the minimum in the following problem for some $L^p$-function $\phi$ defined on the boundary of $B^3$ and having integer degree, i.e. $\int_{\partial B^3}\phi\in\mathbb Z$.

\begin{equation}\label{minpb}
\inf\left\{\int_{B^3} |X|^p: X\in L^p_{\mathbb Z}(B^3, \mathbb R^3),\:\nu_{B^3}\cdot X|_{\partial B^3}=\phi\right\}.
\end{equation}

Note that without the constraint \eqref{star} the minimization in \eqref{minpb} yields the minimum $X\equiv 0$ regardless of the choice of $\phi$. 
We recall the meaning of $\nu_{B^3}\cdot X|_{\partial B^3}$ in Definition \ref{bdryval} and we show the existence of a minimizer for \eqref{minpb} in Lemma \ref{existsmin}. This existence result depends on \cite{PR1} and the fact that in general it is nontrivial is implied by the properties of the trace present in \cite{P3}. The main result of the present work is the following:

\begin{theorem}\label{mainthm}
 Let $p\in]1,3/2[$, and let $X\in L^p(B^3, \mathbb R^3)$ be a minimizer. Then $X$ is locally H\"older-continuous away from a locally finite set $\Sigma\subset B^3$.
\end{theorem}

In the future work \cite{P4} we will address the regularity up to the boundary.

\subsection{Using Calculus of Variations to construct $U(1)$-bundles}
In \cite{PR1} a first step was achieved towards the study of weak bundles with topologically nontrivial singularities. We recall here the main ideas of that approach.\\
First of all, recall that $U(1)$-line bundles over $2$-dimensional surfaces are classified up to isomorphism by their first Chern class $c_1$, which for an $U(1)$-bundle $P$ over a compact surface $\Sigma$ is expressible via Chern-Weil theory as
$$
c_1(P)=\int_\Sigma F_A\in2\pi\mathbb Z\equiv H^2(\Sigma,\mathbb Z),
$$
where $F_A$ is the curvature of any connection $A$ on $P$ (see \cite{Zhang}). By identifying the Lie algebra $u(1)$ with $\mathbb R$, we can identify $F_A$ with an $\mathbb R$-valued $2$-form on $\Sigma$. In the ``supercritical'' dimension $3$, a $2$-form corresponds to a curvature if it assings integer volume to each closed surface (that integer corresponds to the $c_1$ of a line bundle restricted to the surface). By identifying $2$-forms with $1$-vectors in $3$ dimensions, we arrive back at our definition of vector fields with integer fluxes.\\

The idea started by \cite{rivkess} was to study an energy minimization problem on bundles defined weakly as sketched above, the hope being that the minimizers would conserve some nontrivial information expressible in terms of $c_1$. 
\begin{definition}[\cite{rivkess}]\label{definitionFZabel}
 Let $\Omega\subset\mathbb R^3$ be an open domain. We say that an $L^p$-integrable $2$-form $F$ on $\Omega$ a \emph{curvature of a weak line bundle with group $U(1)$}, if for all $x\in\Omega$ and for almost all $r>0$ such that $B(x,r)\subset\Omega$, there holds
$$
\int_{\partial B(x,r)}i^*F\in\mathbb Z,
$$
where $i:\partial B(x,r)\to\mathbb R^3$ is the inclusion map. We denote by $\mathcal F_{\mathbb Z}^p(\Omega)$ the class of such $F$.
\end{definition}
The above definition furnishes a definition of bundles in terms of slices on spheres, and a suggestive parallel can be made with the theory of scans as in \cite{HR},\cite{HR1}. A first consequence of this parallel is the idea leading to the proof of weak compactness in \cite{PR1}.\\

The main novelty of the above kind of definition in comparison with previous contributions, is that no assumption is made a priori regarding the existence of an underlying topological bundle structure, and one only assumes the existence of a curvature form respecting the Chern class constraint. This is the natural setting in which to study regularity problems, and in which to construct new bundles by minimizing the energy. In such way we ``leave the minimizer free'' to find the ``most competitive'' singularities, instead of forcing them on it at the beginning.\\

Given an $L^p$-integrable form $\phi$ on $\partial\Omega$, consider the following minimization problem which is just the translation of \eqref{minpb} into the language of differential forms:
\begin{equation}\label{minpb2}
\inf\left\{\int_\Omega|F|^pd\mathcal H^3:\:F\in\mathcal F_{\mathbb Z}^p(\Omega), i^*_{\partial\Omega}F=\phi\right\}. 
\end{equation}
The first positive result obtained in \cite{PR1} was the fact that $\mathcal F_{\mathbb Z}^p(\Omega)$ is closed under weak convergence. The precise statement of the result is as follows.
\begin{theorem}[\cite{PR1}, Main Theorem]\label{closureweakbdl}
Suppose $F_n\in\mathcal F_{\mathbb Z}^p(\Omega)$ converge weakly to a $2$-form $F$. Then $F\in\mathcal F_{\mathbb Z}^p(\Omega)$.
\end{theorem}
The definition of the boundary value $i^*_{\partial\Omega}F=\phi$ is given in Section 5 of \cite{P3}:
\begin{definition}[\cite{P3}]\label{bdryval}
 Suppose $F\in\mathcal F_{\mathbb Z}^p(B^3)$, and suppose that $\phi$ is a $2$-form in $ L^p(\partial B^3)$ with $\int_{\partial B^3}\phi\in\mathbb Z$. Define the $2$-forms $F(\rho):=T_\rho^* F$, where $T_\rho:\partial B^3\to B^3$ is given by $T_\rho(\sigma)=(1-\rho)\sigma$ for $\rho\in]0,1[$. We say that $F\in\mathcal F_{\mathbb Z,\phi}^p(B^3)$ if $\lim_{\rho\to 0^+}d(F(\rho),\phi)=0$, where the distance $d$ is as in \cite{PR1}, \cite{P3}.
\end{definition}
It was shown in \cite{P3} that 
\begin{proposition}\label{boundaryvalconserved}
 $F_n\in\mathcal F_{\mathbb Z,\varphi_n}^p$ and $F_n\stackrel{L^p}{\rightharpoonup}F$ implies that $\varphi_n\stackrel{d}{\to}\phi$ for some $L^p$-form $\phi$ of integer degree, and that $F\in\mathcal F_{\mathbb Z,\phi}^p$. 
\end{proposition}
We also recall that forms $F$ correspond to vector fields $X$ via the formula 
$$
F_p(U,V)=X_p\cdot (U\times V)\quad\text{for all }U,V\in\mathbb R^3.
$$
The above results imply the existence of minimizers:
\begin{lemma}\label{existsmin}
 If $\phi$ is a $2$-form in $L^p$ [respectively, up to Hodge star duality with respect to the standard metric, an $L^p$-function] on $\partial B^3$ having integer degree and the Definition \ref{bdryval} [respectively, its translation for vector fields] is used for the boundary value, then the minimum is achieved in Problem \eqref{minpb2} [resp. \eqref{minpb}].
\end{lemma}
\begin{proof}
 We give the proof in the language of forms. Consider a minimizing sequence $F_i\in\mathcal F_{\mathbb Z,\phi}^p(B^3)$ and extract a weakly convergent subsequence, which we label in the same way, abusing notation. We denote by $F_\infty$ the limiting $L^p$-form. From Theorem \ref{closureweakbdl} we know that $F_\infty\in\mathcal F_{\mathbb Z}^p(B^3)$. Using Proposition \ref{boundaryvalconserved} we deduce that $F_\infty\in\mathcal F_{\mathbb Z,\phi}^p(B^3)$. Thus $F_\infty$ is the desired minimizer.
\end{proof}

Our Main Theorem \ref{mainthm} implies that a minimizer for the above problem, whose existence is proved above, actually gives an usual bundle (defined outside isolated points) locally inside the domain. Therefore the program of constructing weak bundles by variational methods works in the abelian case.

\subsection{Future steps: nonabelian bundles}
The broader motivation for our work with abelian bundles is that they show some of the features of nonabelian (e.g. $SU(2)$-) bundles, without the complications due to the nonabelianity of the group. The next step after the proof of regularity present here, is indeed to attack the case $SU(2)$, that is the prototype of nonabelian bundles. In this case the classifying invariant is the second Chern class $c_2\in H^4(M)$, defined for a $4$-dimensional surface, again expressible in terms of curvatures, as
$$
c_2(P)=\int_M\op{tr}(F_A\wedge F_A)\in8\pi^2\mathbb Z.
$$
The above invariant is still present and significant for weak bundles, as shown in \cite{Uhlchern}. The definition of weak bundles in a supercritical dimension (which in this case would be $5$) by slicing \cite{rivkess} can be stated also in this case (see also the treatment of \cite{isobe1}, \cite{isobe2}):
\begin{definition}\label{definitionFZ}
 Let $\Omega\subset\mathbb R^5$ be a domain. We call an $L^2$-form $F$ on $\Omega$ with values in $su(2)$ a representative of a weak $SU(2)$-bundle, if for all $x\in\Omega$ and for almost all $r>0$ such that $B(x,r)\subset\Omega$, there exists a gauge transformation $g\in L^\infty(S^4,SU(2))$ and a $W^{1,2}$-connection $A$ of an $SU(2)$-bundle over $S^4$, such that the restriction $F_{B(r,x)}$ of $F$ to $\partial B(x,r)\simeq S^4$ satisfies
$$
g^{-1} F_{B(r,x)}g=dA+A\wedge A.
$$
We call $\mathcal F_{\mathbb Z}^p(\Omega)$ the class of such $F$.
\end{definition}
The idea is to look at the Yang-Mills energy $\int_{\Omega}|F|^2d\mathcal H^5$ in the above class, and the expectation is to obtain also in that case some analogy with the results for the $U(1)$-case considered here. Definitions \ref{definitionFZabel} and \ref{definitionFZ} coincide after replacing $U(1),u(1),p$ by $SU(2),su(2),2$, because in the abelian case $g^{-1}Fg=F$ and $A\wedge A=0$, while in $F=dA$ the regularity of such $A$ directly follows from that of $F$ after applying Fubini's theorem. \\
Regularity results analogous to ours are not proved in the $SU(2)$-case, but a hint that they might be true comes from the singularity removal result of \cite{TaoTian}. Regarding the study of nonabelian bundles supercritical dimension, see also the more general works \cite{tian} and \cite{DoKr}, \cite{DoTh}.\\
The lack of an a priori given smooth structure will make the analogous of our result in the nonabelian case much different from previous works. The interest of the abelian case treated here is that it gives a new hint that the regularity result could be achievable.

\subsection{Relation to the regularity theory for harmonic maps and outline of the paper}
Our regularity result parallels the following result of Schoen and Uhlenbeck (case $p=2$), later extended by Hardt and Lin (for general $p\in]1,\infty[$) regarding minimizing harmonic maps. The result was proved for more general manifolds, but the special case stated here already presents the main difficulties. The more precise description the singularities is due to Brezis, Coron and Lieb.
\begin{theorem}[\cite{SU1},\cite{HL},\cite{BCL}]\label{SU3D}
 Suppose $u:B^3\to S^2$ is a map in $W^{1,2}(B^3,S^2)$ minimizing the $L^2$-norm of its differential. Then $u$ has H\"older-continuous derivative outside a locally finite set $\Sigma\subset B^3$. Moreover, $u$ realizes a nontrivial degree around small spheres centered at each point in $\Sigma$. 
\end{theorem}
The analogy of our problem with the one of harmonic maps is also reflected by the fact that in our case the singularities also encode some topology, i.e. they all have a nontrivial degree. We decided however to prove such description of the singularities in a future work, in order not to make the present article too heavy.\\

The careful reader might be tempted to conjecture that our formulation of a minimization problem in terms of vector fields with integer fluxes can be reformulated in terms of harmonic maps. One could for example consider the minimization problem for the pullback $u^*\omega_{S^2}$ of the volume $2$-form on $S^2$, via a map $u\in W^{1,q}(B^3,S^2)$. Since we are dealing with a $2$-form $\omega$, the natural regularity requirement for $u$ corresponding to $u^*\omega\in L^p$ would then be $u\in W^{1,2p}$. This is encouraged by the observation that the range of exponents corresponding to $p\in]1,3/2[$ is $q=2p\in[2,3[$ and gives precisely the Sobolev spaces for which the weak Jacobian $d(u^*\omega)$ of $u$ is assured to be rectifiable and nontrivial.\\
While at the level of function spaces there is no complication in sight, the problem is that the operation $u\mapsto u^*\omega$ is nonlinear in $du$, since $\omega$ is a $2$-form. Therefore there is no reason to think that the minimizing $u$ should give a minimizing $u^*\omega$, or vice versa. In general it is also not clear that $L^p_{\mathbb Z}$-vector fields are in bijective correspondence with forms $u^*\omega$ 
obtained as above from Sobolev functions $u\in W^{1,2p}(B^3,S^2)$. In the (linear in $du$) case of $1$-forms, such representation is proved in \cite{P1}.\\

Our approach roughly follows the strategy of the regularity theory for harmonic maps. As in the harmonic map regularity proof, we derive and make use of a monotonicity formula and a stationarity formula (cfr \cite{HL} and \cite{Price} with our Section \ref{statmon}). In Section \ref{eregsec} we prove an $\epsilon$-regularity result, in Section \ref{lucksec} we describe an analogous of the Luckhaus lemma \cite{Luckhaus}, which helps showing the sequential compactness of minimizers. Then we proceed to the study of tangent maps and to the dimension reduction in Section \ref{secfinal}.\\

 The techniques and results of sections Sections \ref{eregsec} and \ref{lucksec} are quite different from the approaches that we found in the literature, and might shed a different perspective also on the theory of harmonic map regularity. The main new observation is that the $\epsilon$-regularity can be studied on a simple model if we use the fact that the singularities come with an associated integer (the degree, or flux, of our vector field on small spheres surrounding the singularity).\\
The structure that naturally arises is a weighted graph, having vertices that represent the singularities and edges representing the vector field's flow lines. Reducing to this model is allowed by the strong density result of Kessel, proved in \cite{kessel} and summarized in \cite{rivkess}. Its proof follows the strategy used by Bethuel to prove similar results for Sobolev maps into manifolds \cite{Bethuel} \cite{Bethuel2}. The precise result is the following:
\begin{theorem}[\cite{kessel},\cite{rivkess}]\label{densityRK}
Suppose that $\Omega\subset\mathbb R^3$ is an open set. Call $\mathcal R_{\mathbb Z}^\infty(\Omega)$ the class of vector fields defined and smooth outside a finite set $\Sigma=\{a_1,\ldots,a_n\}\subset\Omega$, and having integer fluxes. Then $\mathcal R_{\mathbb Z}^\infty(\Omega)$ is dense in $L_{\mathbb Z}^p(\Omega)$ with respect to the $L^p$-topology.
\end{theorem}
The approximants to a minimizer (as given by Theorem \ref{densityRK}) correspond then to normal $1$-dimensional currents. We are able to associate a weighted graph to vector fields in $\mathcal R_{\mathbb Z}^\infty$, by applying a decomposition result of Smirnov \cite{smirnov} for normal $1$-dimensional currents (see Theorem \ref{thmc}).\\

The $\epsilon$-regularity theorem is then obtained by a combinatorial reasoning on these graphs. It relies on an elementary minimax result (the famous ``maxflow-mincut'' theorem, \cite{FordFulkerson}). See the scheme \eqref{summary} in the next section for a more precise overview of the proof. The same discretization method is the critical step also in the Luckhaus lemma, in Section \ref{luckhaus}.\\

\textbf{Acknowledgements} I wish to thank Tristan Rivi\`ere for introducing me to the subject, for his encouragement, and for his many helpful comments and suggestions. I also thank Luca Martinazzi for some questions which helped clarifying the initial version of the paper.

\section{The $\epsilon$-regularity theorem}\label{eregsec}
In this section our goal is to prove a so-called $\epsilon$-regularity theorem. This result states that if, for an energy minimizer $X$ on a ball $B$ the energy happens to be small enough, then $X$ has no charges inside a smaller ball:
\begin{theorem}[$\epsilon$-regularity]\label{ereg}
 There exists $\epsilon_p>0$ such that for any minimizer $X\in L^p_{\mathbb Z}$ of the $L^p$-energy if
$B^3_r(x_0)\subset B^3$ and
\begin{equation}\label{enbound}
 r^{2p-3}\int_{B_r(x_0)}|X|^p\; d\mathcal H^3<\epsilon_p,
\end{equation}
then 
\begin{equation}\label{zerodiv}
 \op{div}X=0\quad\text{on }B_{r/2}(x_0).
\end{equation}
\end{theorem}
The main steps of the proof can be summarized as follows (see the scheme \ref{summary} below).
\begin{itemize}
 \item  We first approximate our vector field $X\in L^p_{\mathbb Z}$ strongly in $L^p$-norm by some smoother vector field $\tilde X$ as in Theorem \ref{densityRK}. 
\item To $\tilde X$ we associate a $1$-current $T_{\tilde X}$ in a classical way, and we apply to $T_{\tilde X}$ a decomposition result due to Smirnov \cite{smirnov} (see also the recent development \cite{paolstep}). This result says that a normal current like $T_{\tilde X}$ can be decomposed via a measure (on Borel sets for the weak topology) $\mu_{\tilde X}$ into a superposition of rectifiable integral currents supported on Lipschitz curves starting and ending on the boundary of $T_{\tilde X}$. This result is described in Section \ref{sec:smirnov}.
\item Smirnov's decomposition $\mu_{\tilde X}$ in our case (since the boundary $\partial T_{\tilde X}$ is supported on a discrete set) gives rise to a weighted directed graph $G_{\tilde X}$, by grouping together the curves in the support of $\mu_{\tilde X}$ with the same starting and ending point. These constructions are performed in Section \ref{sec:graphs}.
\item We define a way of perturbing $G_{\tilde X}$ into another graph $G'$. For the underlying vector fields, this corresponds to perturbing $\tilde X$ into a vector field $X'$ that is (not smooth but) still in $L^p_{\mathbb Z}$, and has energy bounded by the energy of $\tilde X$. We call these modifications \emph{elementary operations} (see the definitions at the beginning of Section \ref{sec:graphs}), and we use the same notation for operations on the graph $G_{\tilde X}$ and on the corresponding vector field $\tilde X$.
\item If $\tilde X$ has little energy on a ball $B$, then we can perturb it by elementary operations into another vector field $X'$ as above, and which has no charges inside $B$. This uses the classical ``max flow/min cut'' theorem on the graph $\tilde G$ (see Section \ref{sec:endpf}). 
\item Finally, as the vector fields $\tilde X$ approximate better and better the minimizer $X$, since $p>1$ we can apply the results of \cite{PR1} and extract a subsequence of the perturbed $X'$ that converge weakly to a competitor for $X$. The comparison of $X$ with the competitor gives a contradiction unless $X$ has no charges in $B$, proving the result (see Section \ref{comparisonarg}).
\end{itemize}
\begin{equation}\label{summary}
\xymatrix{
*++[F]{X\in L^p_{\mathbb Z}} \ar@/^2pc/[r]^-*+\txt{approxi-\\mation}&*+[F-:<3pt>]{\tilde X\in\mathcal R_\infty}\ar@{<->}[r] &*+[F-:<3pt>]{T_{\tilde X}}\ar@/^/[drr]^-*+\txt{Smirnov's \\ decomposition}&&\\
*++[F=]{\txt{competitor}}&*+[F-:<3pt>]{X'\in L^p_{\mathbb Z}}\ar@/^2pc/[l]&&&*+[F-:<3pt>]{\mu_{\tilde X}\txt{\\measure on\\rect. lip. curves}}\ar[d]\\
&*+[F-:<3pt>]{\txt{ \\perturbed\\graph}\quad G'}\ar[u]&&&*+[F-:<3pt>]{\txt{ \\weighted\\directed graph}\quad G_{\tilde X}}\ar[lll]_-*+\txt{elementary\\ operations}\\
}
\end{equation}

\subsection{Smirnov's decomposition of $1$-dimensional normal currents}\label{sec:smirnov}
We build our constructions upon Smirnov's decomposition result for $1$-dimensional normal currents \cite{smirnov}. In order to state the results that we use, we need some preliminaries.\\
\begin{definition}
 A $1$-current $T$ in $\mathbb R^3$ is called an \emph{elementary solenoid} if there exists a $1$-Lipschitz function $f:\mathbb R\to\mathbb R^3$ with $f(\mathbb R)\subset \op{spt}(T)$, such that $f,T$ satisfy
\begin{eqnarray*}
 T&=&\mathcal D-\lim_{T\to\infty}\frac{1}{2T}f_\#\overrightarrow{[-T,T]},\\
\mathbb M(T)&=&1.
\end{eqnarray*}
\end{definition}
In the spirit of the above definition, we can identify an oriented Lipschitz curve with a $1$-dimensional rectifiable current. We call $\mathcal C_\ell$ the set of all oriented curves of length $\leq\ell$, which we endow with the weak topology. All measures on paths described in this section will be positive, $\sigma$-finite measures, Borel with respect to the weak topology. The corresponding integrals are understood in the weak sense, i.e. 
$$
S=\int_{\mathcal C_\ell}Rd\mu(R)\text{ is the current defined by }S(\phi)=\int_{\mathcal C_\ell}R(\phi)d\mu(R)\text{ for }\phi\in\mathcal D^1(\mathbb R^3). 
$$
\begin{definition}\label{def:totdec}
 We say that a $1$-current $T$ is \emph{decomposed} into currents lying in a set $J\subset \mathcal D_{1,loc}(\mathbb R^3)$ if there is a Borel measure $\mu$ supported on $J$ such that
\begin{eqnarray*}
 T&=&\int_JRd\mu(R),\\
||T||&=&\int_J||R||d\mu(R).
\end{eqnarray*}
$T\in\mathbb N_{1,loc}(\mathbb R^3)$ is \emph{totally decomposed} if the same $\mu$ also decomposes the boundary:
\begin{eqnarray*}
 \partial T&=&\int_J\partial R d\mu(R),\\
||\partial T||&=&\int_J||\partial R||d\mu(R).
\end{eqnarray*}
\end{definition}
\begin{theorem}
 If $\ell>0$, $T\in \mathcal D_1(\mathbb R^3),\partial T=0$, then $T$ can be decomposed into elements of $\mathcal C_\ell$, with a measure $\mu$ of total mass $\mathbb M(T)/\ell$. Moreover, the following relations hold in the sense of measures:
\begin{eqnarray*}
 \frac{2}{\ell}||T||&\geq&\int_{\mathcal C_\ell}||\partial R||d\mu(R),\\
\frac{1}{\ell}||T||&=&\int_{\mathcal C_\ell}\delta_{b(R)}d\mu(R)=\int_{\mathcal C_\ell}\delta_{e(R)}d\mu(R),
\end{eqnarray*}
where $b(R), e(R)$ are the start and end points of $R$, respectively.
\end{theorem}
We do not use the above theorem, but we cite it because in \cite{smirnov} it is the basic ingredient for the next two theorems, which we rely upon. Using Birkhoff's theorem (in the appropriate setting), Smirnov proves the following decomposition result.
\begin{theorem}
 $T\in \mathcal D_1(\mathbb R^3),\partial T=0$, then $T$ can be decomposed in elementary solenoids.
\end{theorem}
For the case $\partial T\neq 0$ there holds instead:
\begin{theorem}\label{thmc}
 If $T\in\mathbb N_{1,loc}(\mathbb R^3)$ then $T$ can be decomposed as follows:
\begin{eqnarray*}
 T&=&P+Q, \\
||T||&=&||P||+||Q||,\\
\partial T&=&\partial Q,\partial P=0.
\end{eqnarray*}
moreover $Q$ can be totally decomposed into simple curves of finite length, i.e. into elements of $\mathcal C_\infty:=\cup_{\ell>0}\mathcal C_\ell$.
\end{theorem}
\begin{rmk}\label{obs}
We now note some facts that follow easily from the constructions of Smirnov, but are not explicitly stated in his paper:
\begin{enumerate}
\item\label{1obs} In the total decomposition of $Q$ above, the paths have in general unbounded (finite) lengths, but almost all of them (w.r.t. the decomposing measure $\mu$) have $b(R), e(R)$ on the support of $\partial T=\partial Q$. 
\item\label{2obs} If $T$ corresponds to a regular vector field (i.e. for all test forms $\omega$, $T(\omega)=\int\omega(X)d\mathcal L^3$ and $X$ is regular), then the paths are composed of pieces of trajectories of the flow of $X$.
\item\label{3obs} The functions $b, e:\mathcal C_\infty\to\mathbb R^3$ are continuous for the weak topology. In particular, given two Borel sets $A,B\subset \mathbb R^3$, the set of paths 
$$
\{R:\:\mathbb M(R)<\infty, b(R)\in A, e(R)\in B\}
$$ 
is Borel for the weak topology.
\item\label{4obs} Suppose that a $1$-current $T$ decomposes via a measure $\mu$ on the space of $1$-currents. If $\alpha$ is a bounded Borel function on $\mathcal D_1(\mathbb R^3)$, then $\nu=\alpha\mu$ induces by integration a $1$-current $T_\alpha$ that is totally decomposed via $|\alpha|\mu$, and satisfies 
$$
\overrightarrow T_\alpha=\pm\overrightarrow T\text{ and }||T_\alpha||\leq ||\alpha||_{L^\infty(\mu)}||T||. 
$$
Indeed, this is true for step functions $\alpha$, and $L^1$-convergence at the level of the decomposition induces weak convergence at the level of the decomposed currents.
\item\label{5obs} The same result as above holds also in the case of a totally decomposed current $T$, with the analogous inequality holding also for the boundaries:
$$
||\partial T_\alpha||\leq ||\alpha||_{L^\infty(\mu)}||\partial T||.
$$
\end{enumerate}
\end{rmk}
\subsection{Encoding the useful information in a graph}\label{sec:graphs}
For vector fields $  X\in \mathcal R^\infty_{\mathbb Z}(\Omega)$ the decomposition of Smirnov allows to group the integral trajectories of $ X\res\Omega$ according to their start and end points: a generic trajectory could start or end on $\partial \Omega$ or on one of the ``charges'' (i.e. singularities) of $X$. We encode this information in a weighted directed graph (i.e. a graph such that to each edge a positive number called ``weight'' and a direction are assigned). The weights in our encoding graphs keep track of how much of the flux of $X$ is carried by each group of trajectories, and the direction of an edge encodes the direction of the corresponding trajectories. The grouping is done in such a way that there are no flux cancellations within the same group. Thus specifying the flux for a group of trajectories automatically gives a measure of the norm of the restriction of $X$ to those trajectories.

\subsubsection{Elementary operations}
The following kind of operations will be the ones that we perform on our encoding graphs:
\begin{definition}
 An \emph{elementary operation} on a directed weighted graph $G$ consists of multiplying by a factor $\alpha\in[-1,1]$ the weight of an edge, where multiplication of the weight by a negative factor $\alpha<0$ means inverting the orientation and multiplying by $|\alpha|$.\\
We indicate by $G\preceq G'$ the statement that $G$ is achieved from $G'$, after applying finitely many elementary operations.
\end{definition}
We now define the elementary operations on the underlying $X\in\mathcal R_{\mathbb Z}^\infty(\Omega)$. We use the same name because the two definitions correspond to each other in a natural way, as described in Section \ref{sec:assocgraph}.
\begin{definition} Consider $  X\in\mathcal R^\infty_{\mathbb Z}(\Omega)$, which we identify with a current $T=T_{  X}$ as in Remark \ref{obs} \eqref{2obs}, and to which we associate $P,Q$ and a measure $\mu$ totally decomposing $Q$ as in Theorem \ref{thmc}. An \emph{elementary operation} on $  X$ consists in replacing $  X$ by the vector field corresponding to $(T_{  X})_\alpha$ obtained as in Remark \ref{obs} \eqref{5obs}, for some function $\alpha$ that only takes values in $[-1,1]$ and that is piecewise constant on a family of sets defined via $b,e$ as in Remark \ref{obs} \eqref{3obs}.\\
We indicate by $  X\preceq   X'$ the property of $  X$ of being achievable after performing finitely many elementary operations starting from $  X'$.
\end{definition}

\begin{rmk}
\begin{enumerate}
 \item It is immediate form Remark \ref{obs} \eqref{3obs} that $  X\preceq   X'$ implies $||  X||_{L^p}\leq||  X'||_{L^p}$ with strict inequality unless $|\alpha|=1$ in all of our elementary operations.
 \item $\mathcal R^\infty_{\mathbb Z}$ is not invariant under elementary operations, since such operations often create jumps in $  X$. In general also the integer divergence condition is not preserved by these modifications.
\item From Remark \ref{obs} \eqref{5obs} it follows however, that for $  X\in\mathcal R^\infty_{\mathbb Z}\cap L^p(\Omega)$, any elementary operation sends $  X$ to a vector field $  X'\in L^p(\Omega)$ having zero divergence away from the singular set of $  X$.
\end{enumerate}
\end{rmk}

\subsubsection{Grouping trajectories of $X\in \mathcal R_{\infty}\cap L^1(\Omega)$}

\begin{figure}[htp]
\centering
\scalebox{0.5}{\input{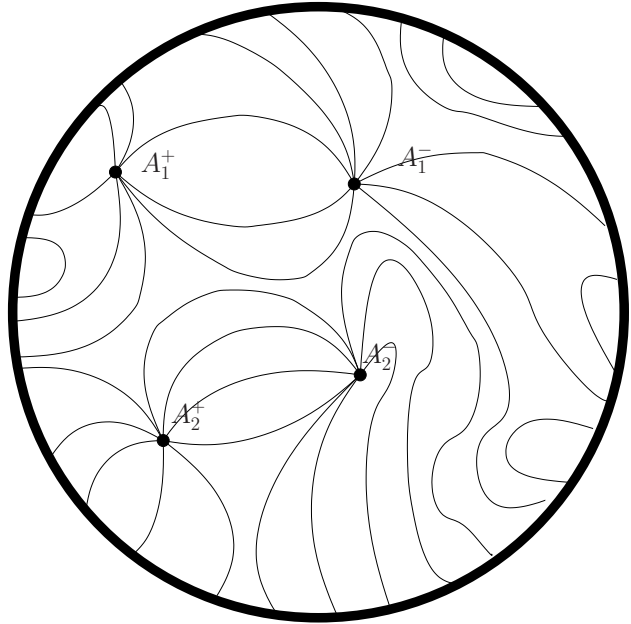}}
\caption{We represent schematically (i.e. we forget for a moment that we are in a $3$-dimensional setting, and we take $\Omega$ to be a ball) the finitely many charges of our vector field $X\in\mathcal R_{\infty}\cap L^1(\Omega)$ as black dots, and some of the supports of the rectifiable currents $R$ of Definition \protect\ref{def:totdec}, as thin lines.}\label{fig:lines}
\end{figure}

Consider $X\in\mathcal R_{\infty}\cap L^1(\Omega)$ and the normal $1$-current $T_X$ as in Remark \ref{obs} \eqref{2obs}.\\
 
Using Theorem \ref{thmc}, we can find a decomposition $T_X=P_X+Q_X$ and a measure $\mu_X$ on $\mathcal C_\infty:=\cup_{\ell>0}\mathcal C_\ell$ that totally decomposes $Q_X$ into finite-length simple paths.\\
Then note that, due to the special structure of $X$, $\partial(T_X\res\Omega)$ is supported on $\partial\Omega\cup\{\text{charges of }X\}$. Also, by the total decomposition property of $Q_X$, there holds
$$
\partial (T_X\res B)=\int_{\mathcal C_\infty}\partial Rd\mu_X(R)=\int_{\mathcal C_\infty}(\delta_{e(R)} - \delta_{b(R)})d\mu_X(R)
$$
and  $b(R),e(R)\in\op{spt}\partial (T_X\res B)$ for $\mu_X$-a.e.$R$, so that we can decompose the set of finite length paths into disjoint Borel sets:
$$
\mathcal C_\infty=C\cup\bigcup_{i,j=0}^nC_{ij},
$$ 
where $\mu_X(C)=0$ and for all $R\in C_{ij}$ there holds
$$
b(R)\in A_i^-,e(R)\in A_j^+,
$$
where
$$
 A_0^\pm:=\partial\Omega\cap\{\op{sgn}(X\cdot\nu_\Omega)=\pm1\}
$$
and
$$
 A_i^\pm,i>0\text{ enumerate the }\pm-\text{charges of }X,\text{ possibly with repetitions.}
$$
By the decomposition theorem \ref{thmc}, if 
$$
C_i^-=\cup_{j=0}^n C_{ij}, C_j^+=\cup_{i=0}^n C_{ij},
$$ 
then 
$$
\mu_X(C_i^+)=\sum_{j=0}^n\mu_X(C_{ij}),\quad \mu_X(C_j^-)=\sum_{i=0}^n\mu_X(C_{ij}),
$$
and for $i>0$ it is clear that $\mu_X(C_i^\pm)$ is equal to the charge of $A_i^\pm$ (see also Figure \ref{fig:lines2}).

\begin{figure}[htp]
\centering
\scalebox{0.5}{\input{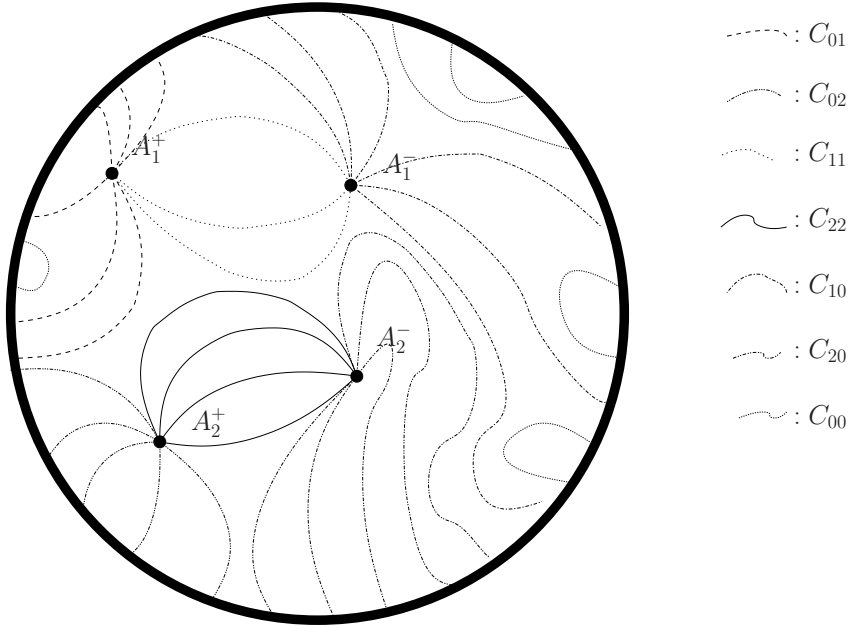}}
\caption{In the example of Figure \protect\ref{fig:lines}, we represent with different patterns the supports of curves belonging to different $C_{ij}$'s. We omit the set $C_{ij}$ if it has $\mu_X(C_{ij})=0$.}\label{fig:lines2}
\end{figure}

\subsubsection{Associating a graph to a vector field}\label{sec:assocgraph}
With the notations of the previous subsection, we associate to $X$ the graph $G_X$ (see Figure \ref{fig:graph2}) which has the following features:
\begin{itemize}
\item has vertices indexed by $A_i^\pm, i=0,\ldots,n$,
\item has a directed edge $A_i^-\to A_j^+$, for all $0\leq i,j\leq n$, unless $\mu(C_{ij})=0$,
\item any edge $A_i^-\to A_j^+$, it has weight $\mu_X(C_{ij})$ assigned to it.
\end{itemize}

\begin{figure}[htp]
\centering
\scalebox{0.5}{\input{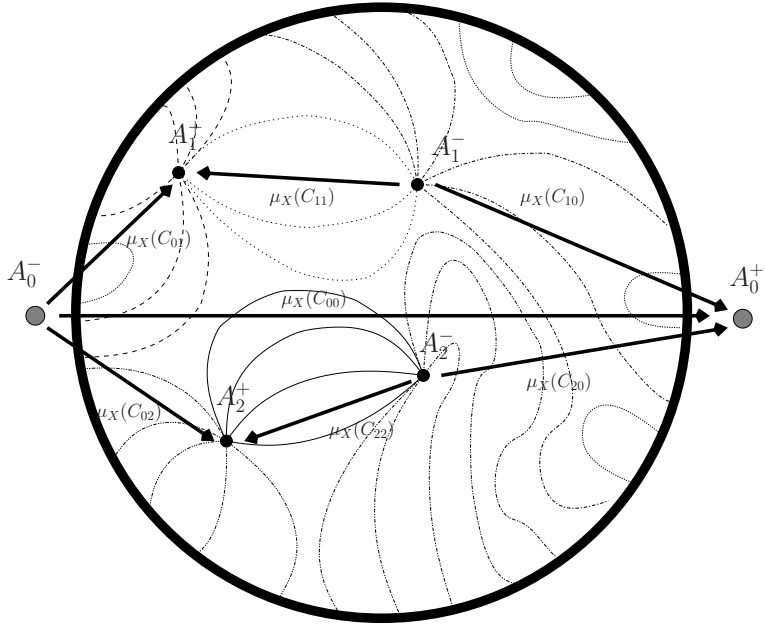}}
\caption{We superpose to the picture of Figure \ref{fig:lines2} the associated graph, where on top of each arrow we also describe its weight. The gray vertices $A_0^+, A_0^-$ correspond respectively to start and end points of curves which lie on the boundary.}\label{fig:graph2}
\end{figure}

Further, if $\bar G\preceq G_X$ then we associate to $\bar G$ a vector field $\bar X\preceq X$ such that $\bar G=G_{\bar X}$, by the following procedure:
\begin{itemize}
 \item Fix a sequence $G_X=G_0\preceq G_1\preceq\cdots\preceq G_N=\bar G$ such that $G_{k+1}$ is obtained from $G_k$ by an elementary operation. We can still identify the vertices of $G_k$ with those of $G_X$. 
\item To each $G_k$ we associate a function $\alpha_k\in L^\infty(\mu_X)$, as follows. We start with $\alpha_0\equiv 1$. For $k>0$ if $G_{k+1}$ is obtained from $G_k$ by multiplying the weight on $A_i^-\to A_j^+$ by $\alpha\in [-1,1]$ then we define $\alpha_{k+1}:=\alpha \chi_{C_{ij}}\alpha_k + \chi_{\mathcal C_\infty\setminus C_{ij}}\alpha_k$.
\item Clearly $\alpha_N\in L^\infty(\mu_X)$ defines an elementary operation on $X$, and so we call $\bar X$ the vector field corresponding to $(T_X)_{\alpha_N}$.
\end{itemize}
\subsection{Proof of the $\epsilon$-regularity}\label{sec:endpf}
\subsubsection{Modifications to eliminate charges in the regular case}\label{sec:remcharges}
In this subsection we restrict to vector fields $X\in\mathcal R_{\mathbb Z}^\infty(\Omega)$ satisfying the conditions of the $\epsilon$-regularity theorem, and we show that we can apply elementary operations decreasing the energy while eliminating the charges of $X$. The main result is as follows.
\begin{proposition}[regular case]\label{regcase}
Suppose that $X\in \mathcal R^\infty_{\mathbb Z}\cap L^p(\hat\Omega)$ and that $\hat\Omega\Supset\Omega$ is such that $\int_{\partial \Omega}|X|<1$ and $\int_{\partial \Omega}X\cdot\nu =0$. Then there exists a second vector field $\bar X\in L^p_{\mathbb Z}(\hat\Omega)$ such that $\bar X\preceq X$ and
\begin{enumerate}
 \item $\bar X=X$ on $\overline{\hat\Omega\setminus \Omega}$,
 \item $||\bar X||_{L^p(\hat\Omega)}<||X||_{L^p(\hat\Omega)}$ and
 \item $(\op{div}\bar X)\res \Omega=0$.
\end{enumerate}
The inequality of point (2) is strict unless $X$ already satisfies point (3).
\end{proposition}
\begin{proof}
The main idea of the proof is to apply elementary operations to $X$, so that we cancel out the charges inside $\Omega$. Because of the above constructions, it is enough to do the corresponding operations on the graph $G_X$ that encodes all the information that we need for the proof.\\
 \textbf{Step 1: structure of the graph $G_X$.} Consider the graph $G:=G_X$ defined in Section \ref{sec:assocgraph}, and call 
\begin{itemize}
 \item $C^+, C^-$ the sets of vertices of $G$ corresponding to the interior charges of a given sign,
\item $\Sigma^\pm$ the sets of vertices of $G$ corresponding to components of $\partial \Omega$ with local charge $\pm$, i.e. $\Sigma^{\pm}=\{A_0^{\pm}\}$.
\end{itemize}
The form of our graph is summarized in the following scheme, where we also indicate names for groups of arrows:

$$
\xymatrix{
\Sigma^+\ar[r]^{\sigma^+}\ar@/^2pc/[rrr] &C^-&C^+\ar[l]_\nu\ar[r]^{\sigma^-}&\Sigma^-}.
$$

The hypothesis $\int_{\partial \Omega}|X|^p<1$ implies that the arrows $\sigma^\pm$ have total weight less than $1$. This will be important in the sequel.\\
\textbf{Step 2: elimination of the singularities.}
We want to keep the arrows in $\sigma^\pm$ fixed, and modify the other arrows via elementary operations so that the modified graph satisfies Kirchhoff's law. This can be done as follows:
\begin{itemize}
 \item We keep (i.e. multiply by $+1$) all the edges which go directly from $\Sigma^+$ to $\Sigma^-$. Since these edges are not affected by the elementary operations done in the rest of the proof, we suppose from now on, without loss of generality, that there are no such edges.
 \item Let's restrict to a connected component of our graph. Suppose first that it has the form drawn above (i.e. it is not degenerate): in this case we can find a maximal Kirchhoff subgraph $K$ connecting $\Sigma^+$ to $\Sigma^-$, in the undirected graph 
$$
\xymatrix{\Sigma^+\ar@{-}[r]&C^-\ar@{-}[r]&C^+\ar@{-}[r]&\Sigma^-}.
$$
By the ``max flow-min cut'' theorem, after subtracting such directed subgraph, the remaining edges make a disconnected graph that has $4$ possible forms (where we keep the orientations as in the original $G$):

\begin{enumerate}
 \item All arrows in $\nu$ have been cut, but there are some edges joining $\Sigma$ to some point charges. These charges correspond to singularities of $X$, for which at least $1/2$-charge flowed from/to $\Sigma$. In particular, since the difference $|\sigma^+|-|\sigma^-|$ is constant during our construction, there must be an even number of such charges. This is not possible because the $\int_\Sigma |X\cdot\nu_\Sigma|d\mathcal H^2$ was assumed to be smaller than $1$.
\item \label{hg2} The whole graph has been used, and we end up without leftover edges of the graph. Then again we see that $\int_\Sigma |X\cdot\nu_\Sigma|d\mathcal H^2$ is prohibited to be smaller than $1$, since in any charge connected to $\Sigma^\pm$, the total wight of the arrows from/to the boundary $\partial \Omega$, is $=\frac12$, and there are at least $2$ such charges.
\item \label{hg1} All arrows $\sigma^-$ have been cut. Then also the arrows in $\sigma^+$ have disappeared after eliminating the maximal flow, again because $|\sigma^+|-|\sigma^-|$ is constant (equal to zero) during these modifications. Thus in this case all arrows outside $\nu$ are canceled. Then we can multiply by zero the remaining arrows: these arrows are of positive total weight since else we reduce to point \eqref{hg2}, which is already excluded. Thus we strictly decrease the $L^p$-norm of $X$.
\item The last case is the ``generic'' one: it could be that after the cut we are left with a graph of the form
$$
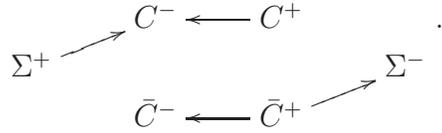
\xymatrix@R0.2pc{&C^-&C^+\ar[l]&\\\Sigma^+\ar[ur] &&&\Sigma^-\\&\bar C^-&\bar C^+\ar[ur]\ar[l]&}.
$$
It is shown in Lemma \ref{genericcase} that in this case it is possible to find another minimal cut that gives a graph as in  \eqref{hg1} or in \eqref{hg2} instead, and we conclude the proof.

\end{enumerate}
The conclusion of this enumeration is that the only possible cases that allow any singularity at all inside $\Omega$ and are compatible with the small boundary energy are the ones corresponding to the case \eqref{hg1} above. Observe that in this case \emph{we are sure to have canceled some edges} i.e. we have decreased the energy of $\tilde X$, as wanted.
\end{itemize}
\end{proof}
\begin{example}
 Consider a regular vector field in $\mathcal R_{\infty}\cap L^p(\Omega)$ that has $5$ singularities, one point having charge $1$ a second point having charge $2$, and the remaining points having charge $-1$ each (see Figure \ref{fig:mcmf}). Suppose that the weights of the edges of the associated graph are as in Figure \ref{fig:mcmf}. We assume from the beginning that $\mu_X$ gives no weight to the curves that both start and end point on the boundary (such curves are anyways not affected by our manipulations). The maximal flow showed on the right corresponds to any of the $3$ minimal cuts on the left. In general, no uniqueness of either the maximal flow or the minimal cut is guaranteed.
\begin{figure}[htp]
\centering
\scalebox{0.5}{\input{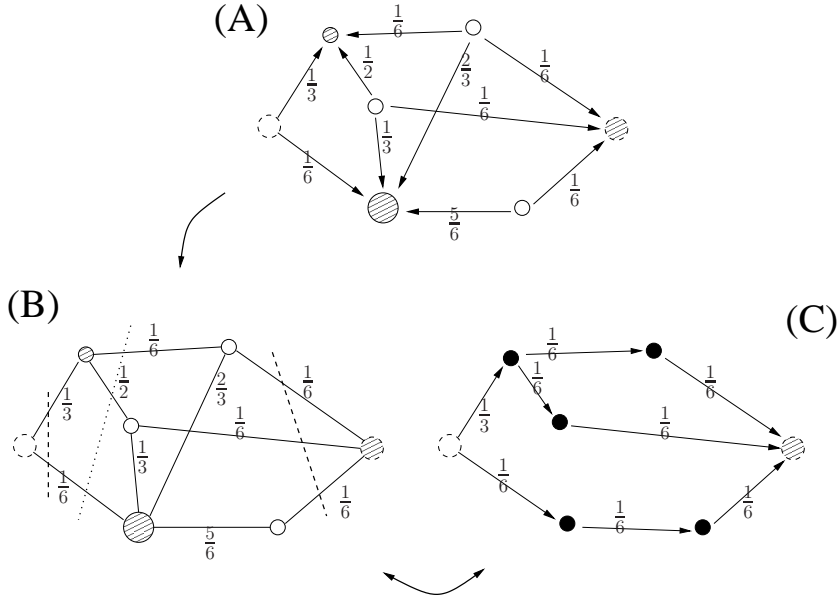}}
\caption{(A): a graph corresponding to a possible vector field $X$ having $6$ charges (of which the one represented by a larger circle is a double one). In the unoriented graph (B), we represent by dashed lines two minimal cuts separating the vertices with dashed boundaries; a non-minimal cut is represented by a dotted line. Observe that the flow through each of the $3$ cuts in (A) is the same, but in (B) the sum of edge capacities is larger for the dotted line. In (C) we show the unique maximal obtained on (B) between the gray vertices.}\label{fig:mcmf}
\end{figure}
In Figure \ref{fig:canccharges} it is shown what happens next, in our manipulations. Once we fix the maximal flow of Figure \ref{fig:mcmf}, we change by elementary operations the flow lines of $X$, ending up with the graph on the left of Figure \ref{fig:canccharges}. Since this represents a flow, i.e. obeys Kirchhoff's law, the curves representing the modified vector field $\bar X$ are concatenated, i.e. that they all start and end on the boundary. This concatenation is ``automatically done'' by Smirnov's decomposition, since the associated current $T_{\bar X}$ is totally decomposed (see Definition \ref{def:totdec}). The ``canceled flow'' on the right of the figure, gives a measure of the amount of $L^p$-norm of $X$ gained this way.\\

We must point out that the $L^p$-energy improvement in passing from $X$ to $\bar X$ depends also from factors not captured by the graph $G_X$ itself, namely on the lengths and concentrations of the curves decomposing the associated current $T_X$. But for our purposes a subtler analysis along these lines is not needed.
 \begin{figure}[htp]
\centering
\scalebox{0.5}{\input{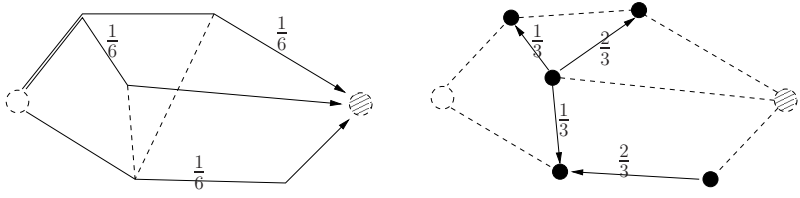}}
\caption{Continuing with the example of Figure \protect\ref{fig:mcmf}, we show schematically on the left what remains after the cancellation of the charges (in terms of the associated graphs the three arrows of weight $\frac{1}{6}$ actually are substituted by just one arrow of weight $\frac{1}{2}$, but we drew the picture to suggest that a procedure of ``concatenating arrows'' is actually underlying the operation). On the right we show the flow that results after removing the maximal flow graph out of the initial graph. In our charge removal procedure, we diminish the weights of our graph by the amounts in the right picture, so in this particular $\bar X$ has a smaller energy than $X$.}\label{fig:canccharges}
\end{figure}
\end{example}

\begin{lemma}\label{genericcase} Under the hypotheses of Proposition \ref{regcase} on $X$, suppose that a connected component of the associated graph $G_X$ has the form 
$$
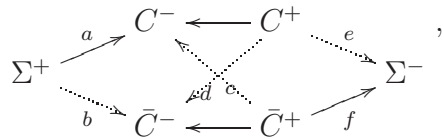
\xymatrix@R0.3pc{&C^-&C^+\ar[l]\ar@{.>}[ddl]^c\ar@{.>}[dr]^e&\\\Sigma^+\ar[ur]^a\ar@{.>}[dr]_b &&&\Sigma^-\\&\bar C^-&\bar C^+\ar[ur]_f\ar[l]\ar@{.>}[uul]^d&},
$$
where a minimal cut is given by the arrows in $b,c,d,e$. Then another minimal cut is given by the arrows in $a,b$.
\end{lemma}
\begin{proof}
The fact that $a,b$ give a cut is clear from the above diagram. We must show that such cut is a minimal one.\\

We indicate by $|x|$ the total flow through the arrows of the group labelled by $x$. First of all observe that by the zero total flux and small boundary energy hypotheses on $X$, 
$$
|a|+|b|=|e|+|f|<\frac12,
$$ 
therefore, being $b,c,d,e$ a minimal cut, by comparison with the above cut we obtain
$$
|b|+|c|+|d|+|e|<\frac12.
$$ 
This implies that \emph{the total number of charges contributing to the vertices $C^+$ is the same as the number of charges contributing to $C^-$, and similarly for $\bar C^+,\bar C^-$}. Indeed, suppose for contradiction that the numbers of charges contributing to $C^+,C^-$ were not equal. Then the total flow $|a|+|c|+|d|+|e|$ would be $\geq1$, and this would contradict the fact that $|a|$ and $|c|+|d|+|e|$ are both $<\frac12$.\\
By the consideration in italics above, we obtain that
$$
|a|+|d|=|c|+|e|,\quad |b|+|c|=|d|+|f|.
$$
Therefore, by definition of a minimal cut
$$
|b|+|c|+|d|+|e|\leq|a|+|b|
$$
and this gives, using the previous computations,
$$ 
|a|\geq|c|+|d|+|e|=|a|+2|d|,
$$
so $|d|=0$ and the above inequalities are actually equalities, as wanted.
\end{proof}

\subsubsection{The proof of $\epsilon$-regularity}\label{comparisonarg}
\begin{proof}[Proof of Theorem \ref{ereg}]
 First of all, we may reduce to the case where the $\epsilon$-regular ball $B(x_0,r)$ of the theorem is the unit ball $B=B(0,1)$, since the estimates and the function spaces considered are invariant under homotheties and translations of $\mathbb R^3$. We call $\hat\Omega$ the image of the initial $B_1$ under this transformation.\\
\textbf{Step 1: fixing a small energy sphere.} We claim that, for any small $\epsilon>0$, we can find a positive measure set of radiuses $\rho>1/2$ such that $\int_{\partial B_\rho}|X|^p<2\epsilon_p$. Indeed, if the opposite estimate would hold for a.e. $\rho>1/2$, then we would obtain 
$$
\int_B|X|^p\geq\int_{1-\epsilon}^1\int_{\partial B_\rho}|X|^p >\epsilon_p,
$$
therefore 
$$
\int_B|X|^pd\mathcal H^3\geq (1-\epsilon)\epsilon_0,
$$
and this contradicts our assumption for $\epsilon$ small enough. Now from the above boundary energy bound by $\epsilon_p$ we get via H\"older's inequality the following bound
$$
\int_{\partial B_\rho}|X|\leq(2\epsilon_p)^{\frac{1}{p}}\left[\mathcal H^2(S^2)\right]^{\frac{p-1}{p}},
$$
and we choose $\epsilon_p$ such that the right hand side is equal to $1$. This gives the small boundary energy condition as in Proposition \ref{regcase}, and the zero flux condition follows from the definition of $L^p_{\mathbb Z}(B)$ and from the inequality $|X\cdot \nu_{B_\rho}|\leq |X|$.\\
\textbf{Step 2: passing to the approximants.} We know that there exist $\tilde X_k\in\mathcal R_{\mathbb Z}^\infty(\hat\Omega)$ that converge to $X$ in $L^p$-norm. From the construction leading to this approximation it is clear that we can also further impose the convergence 
$$
\tilde X_k|_{\partial B_\rho}\stackrel{L^p}{\to} X|_{\partial B_\rho},
$$
therefore for $k$ large enough, $\tilde X_k$ satisfies the properties required in Proposition \ref{regcase}. Applying this proposition, we thus obtain $\bar X_k\in L_{\mathbb Z}^p(\hat\Omega)$ which are equal to $\tilde X_k$ outside $B_\rho$ and satisfy $||\bar X_k||_{L^p(\hat\Omega)} \leq ||\tilde X_k||_{L^p(\hat\Omega)}$ (with strict inequality if $(\op{div}\tilde X_k)\res B_\rho\neq 0$) and $(\op{div}\bar X_k)\res B_\rho=0$.\\
\textbf{Step 3: a divergence-free competitor.} By weak compactness of $L^p_{\mathbb Z}(\hat\Omega)$ it follows that a subsequence of the $\bar X_{k'}$ converges weakly to some $\bar X\in L^p_{\mathbb Z}(\hat\Omega)$. The zero divergence condition passes to weak limits, so $\op{div}\bar X=0$ on $B_\rho$. By sequential weak lowersemicontinuity of the norm, we also deduce
$$
||\bar X||_p\leq\liminf_{k'}||\bar X_{k'}||_p\leq \liminf_k||\tilde X_k||_p=||X||_p.
$$
Since $X$ was a minimizer, all the above inequalities must actually be equalities. We also observe that since the sequence $\bar X_{k'}$ converges both weakly and in norm, it must converge also strongly, to $\bar X$. By examining the definition of elementary operations we also observe that the inequality $|\bar X_k|(x)\leq |\tilde X_k|(x)$ holds almost everywhere for all $k$, and from it and the a.e. convergence it follows that the same inequality holds also in the limit. Since both $\bar X$ and $X$ are minimizers it further follows that $|\bar X|(x)=|X|(x)$ almost everywhere.\\

\textbf{Step 4: $X$ is also divergence-free.} We use the classical regularity theory, namely Lemma \ref{minforms} (which applies since $\op{div}\bar X=0$) and Proposition \ref{regularityresult} to deduce that $\bar X$ is H\"older-continuous in the interior of $B_{r/2}$. It then follows that also $\op{div}X=0$ on $B_{r/2}$, since in this case $X\in L^\infty_{\mathbb Z}(B_{r/2})$. Indeed, using Theorem \ref{densityRK} it follows that $X$ can be approximated by vector fields in $\mathcal R_{\mathbb Z}^\infty(B_{r/2})$ in the strong norm in $L^q$ for $q>3/2$. But for such exponents the vector fields in $\mathcal R_{\mathbb Z}^\infty\cap L^q(\Omega)$ are all smooth (and in particular divergence-free, since the divergence is concentrated at their singular points). Thus by approximation also $X$ is divergence-free.
\end{proof}

\subsection{A classical consequence: $C^{0,\alpha}$-regularity}

From Theorem \ref{ereg}, using an extension by Peter Tolksdorff (and Christoph Hamburger) of the regularity theory first developed by Karen Uhlenbeck, it is relatively straightforward to prove the following extension of it: 
\begin{theorem}[H\"older version of the $\epsilon$-regularity]\label{hereg}
 If $X\in L^p_{\mathbb Z}$ is a minimizer then we can find an $\epsilon_p>0$ such that if on
$B^3_r(x_0)\subset B^3$ the vector field $X$ satisfies 
\eqref{enbound}
then on $B_{r/2}(x_0)$ the vector field $X$ is $\alpha$-H\"older, with $\alpha$ depending only on $p$ and with the H\"older constant of $X|_{B_{r/2}}$ depending only on $p$ and on $||X||_{L^p(B_r)}$.
\end{theorem}

In order to prove the above theorem, we use the conclusion that $\op{div} X=0$ of Theorem \ref{ereg} and the Euler equation of the functional $\int_{\Omega} |X|^p $ to reduce to the by now classical regularity result for systems of equations due to the above cited authors. The main heuristic idea in play here is that roughly ``$\op{div} X=0$ implies that $X=\nabla f$ for some $W^{1,p}_{loc}$-function $f$''. 

In order to use this idea while still keeping rigorous, we use the formulation of our minimization problem in terms of differential $2$-forms $\omega$ instead of vector fields $X$.

\begin{lemma}\label{minforms}
 The condition that a vector field $X\in L^p_{\mathbb Z}(\Omega)$ minimizes the $L^p$-energy and satisfies $\op{div}X=0$ implies that the associated $2$-form $\omega\in \mathcal F_{\mathbb Z}^p(\Omega)$ satisfies locally in the sense of the distributions the following equations:
$$
\left\{\begin{array}{l}
        d\omega=0\\ \delta\left(|\omega|^{p-2}\omega\right)=0.
       \end{array}
\right.
$$
\end{lemma}
\begin{proof}
 The first equation is a trivial translation of $\op{div}X=0$ in our new setting. The second one is the Euler equation, and can be directly obtained from the requirement that $\omega$ be minimizing, by using the perturbations $\omega\mapsto \omega+\epsilon d\phi$, for $\phi\in C^\infty_0(\wedge^2\Omega)$ and taking the derivative in $\epsilon$ at $\epsilon=0$. Since $d\psi$ is exact, it easily follows that the perturbed form is still in $\mathcal F_{\mathbb Z}^p(\Omega)$.
\end{proof}

With the result of the above lemma, we are exactly in the setting of \cite{Uhlenbeck}, except that that article treats the case $p>2$, while we are interested in the case $1<p<3/2$.\\

Luckily, the result of \cite{Uhlenbeck} was extended in \cite{Tolksdorf}, \cite{Tolksdorf2} to the case $1<p<2$. The article of Tolksdorf considers only the ``basic case'' where the equations concern a differential of a function instead of the generalization of exact differential forms described by Uhlenbeck, but the setting in which Tolksdorf proves regularity can be translated without much effort into the one of Uhlenbeck, and the techniques present there are not affected by the translation. Later on Hamburger \cite{Hamburger} partially recovers the approach of Uhlenbeck for the whole range of exponents of Tolksdorf, but this work deals with homological minimizers (i.e. minimizers with respect to perturbations as in the proof of Lemma \ref{minforms} that keep the comparison forms in our class $\mathcal F_{\mathbb Z}^p(\Omega)$) instead of just using the Euler equations as the other two works. In particular, the range of exponents $p\in]1,2[$ is recovered from the range $p>2$ via a duality argument where 
the requirement of dealing with minimizers is involved. Due to these considerations, we can safely state the following version of these regularity results. A recent treatment of the regularity for differential forms including the case needed here is \cite{BS}.
\begin{proposition}[\cite{Tolksdorf},\cite{Hamburger},\cite{BS}]\label{regularityresult}
 If $\omega\in L^p(\wedge^2\Omega)$ satisfies the equations of Lemma \ref{minforms} in the weak sense, then $\omega$ is $\alpha$-H\"older, with $\alpha$ depending only on $p$ and with the local H\"older constant of $X|_{B_{r/2}}$ depending only on $p$ and on $||X||_{L^p(B_r)}$ for any ball contained in $\Omega$. 
\end{proposition}

From the above lemma and the proposition, it is straightforward that Theorem \ref{hereg} holds.

\section{For minimizers $X$, weak convergence implies strong convergence}\label{luckhaus}
In this section we prove the following compactness result:
\begin{theorem}\label{weakstrong}
 Suppose $X_k\in L^p_{\mathbb Z}(B)$ are minimizers of the $L^p$-energy, and that $X_k\rightharpoonup X$ weakly in $L^p$. Then $X$ is also a minimizer and $X_k\to X$ also $L^p$-strongly on any ball $B(0,r), r<1$. In particular, any sequence of minimizers of bounded energy has a strongly convergent subsequence.
\end{theorem}
It is a classical result that strong convergence can fail while weak convergence holds, only if some energy is lost in the limit. Thus, it remains to prove that the energy of $X$ on $B_r$ is not lower than the limit of the energies of the $X_k$ on the same ball. The fact that any $X$ obtained as a strong limit of minimizers is a minimizer itself follows from the strong local convergence.\\

The idea of the proof is to introduce a small parameter $\epsilon>0$ and to construct an interpolant $\tilde X_k\in L^p_{\mathbb Z}(B)$ that equals $X_k$ on $B\setminus B_{r+\epsilon}$ and $X$ inside $B_r$, in such a way that the energy of $\tilde X_k$ in the small spherical shell $B\setminus B_{r+\epsilon}$ goes to zero as $\epsilon\to 0$. This allows us, using the minimization property of $X_k$, to bound from above the energy of $X_k$ on $B_r$, by the energy of $X$ on the same ball.\\

For the proof of the $\epsilon$-regularity, it is enough to be able to do the constructions for vector fields in $\mathcal R_\infty$. The interpolation construction faces again a problem related to possibility that (the approximant of) $X-X_k$ have some singularities in the small shell $B_{r+\epsilon}\setminus B_r$. We deal with this situation again by choosing shells where on the boundaries $X-X_k$ does not have large energy for $k$ large, and by applying the singularity removal operations of Proposition \ref{regcase} from the $\epsilon$-regularity proof. After these elementary operations, we are reduced to an easier situation (see Figure \ref{fig:weakstrong1}), where the curves of the Smirnov decomposition of our vector fields all move from one boundary of the shell to the other. In this simpler case, the interpolation can be done via an auxiliary function $f$ satisfying a Neumann boundary value problem in the shell, and the scaling of the classical energy bounds as the thickness of the shell vanishes (see 
Lemma \
ref{
scaling}), are strong enough for our purposes.

\subsection{Interpolant construction in the regular case}\label{lucksec}
The result on the existence of the interpolants that we need is the following.
\begin{proposition}\label{interpolant}
 There exists a constant $C$ depending only on our exponent $p$ from above, such that the following holds. For any numbers $R$ and $\epsilon$ such that $R>1+\epsilon>1$, for any $Y\in \mathcal R_\infty(B_R)$ having zero flux through $\partial B_{1+\epsilon}$ and through $\partial B_1$, having no singularities lying on these two boundaries, and satisfying 
$$
\int_{\partial B_r}|Y|d\mathcal H^2<\frac{1}{2},
$$
for $r=1$ and for $r=1+\epsilon$, there exists another vector field $\bar Y\in \mathcal R_\infty(B_R)$, such that
\begin{itemize}                                                                                                                                             \item $\bar Y= Y$ on $B_1$,
\item $\bar Y= 0$ outside $B_{1+\epsilon}$,
\item $||\bar Y||_{L^p(B_{1+\epsilon}\setminus B_1)}\leq ||Y||_{L^p(B_{1+\epsilon}\setminus B_1)}+C \epsilon^{-\frac{1}{p}}||Y||_{L^p(\partial B_1)}$.
\end{itemize}
\end{proposition}

\begin{proof}[Proof of Proposition \ref{interpolant}]
Consider the total decomposition $\mu$ of the current $T_Y$ associated to $Y$. In order to prove Proposition \ref{interpolant}, we proceed in two steps. In the first one (see Section \ref{sec:remcharges}), we apply some elementary operations $Y|_{B_{1+\epsilon}\setminus B_1}$, obtaining a new vector field $Y_1$ such that
\begin{itemize}
 \item $\bar Y_1:=\chi_{B_{1+\epsilon}\setminus B_1}Y_1+(\chi_{B_R}-\chi_{B_{1+\epsilon}\setminus B_1})Y$ still belongs to $L^p_{\mathbb Z}(B_R)$,
 \item $||\bar Y_1||_{L^p(B_R)}\leq||Y||_{L^p(B_R)}$,
 \item $\op{div}Y_1=0$ in the interior of $B_{1+\epsilon}\setminus B_1$,
 \item $\mu_{Y_1}\res S=\mu_{Y|_{B_{1+\epsilon}\setminus B_1}}\res S$ where the Borel (for the weak topology) set $S$ consists of the $1$-currents $R$ having boundary on $\partial B_{1+\epsilon}\cup\partial B_1$.
\end{itemize}
In the second step, we modify the currents $R\in S$ (that up to now were untouched by our construction). We apply an elementary operation in which we cancel (i.e. multiply by $0$) the $R$'s with both boundaries on $\partial B_{1+\epsilon}$, and we let the others unchanged. Then we consider (identifying the current $T_2$ with a vectorfield $Y_2$)
$$
T_2=Y_2:=\int_{S'}R d\mu_{Y_1}(R),
$$
where $S'$ are the currents corresponding to lipschitz curves with one end on $\partial B_{1+\epsilon}$ and the other one on $\partial B_1$. It follows $Y_2\preceq Y_1$ and we see that $Y_2$ is an $L^p$-vector field, and since $\mu_{Y_2}$ totally decomposes $Y_2$, there holds $\op{div}Y_2=0$ on $B_{1+\epsilon}\setminus B_1$. The elementary operations decrease the $L^p$-norms of the boundary values, thus 
\begin{equation}\label{esty2}
\int_{\partial B_1}|Y_2|^pd\mathcal H^2\leq \int_{\partial B_1}|Y|^pd\mathcal H^2.
\end{equation}
We now are in a situation where on one hand
$$
\partial T_2=(\nu\cdot Y_2|_{\partial B_{1+\epsilon}})\mathcal H^2\res\partial B_{1+\epsilon} - (\nu\cdot Y_2|_{\partial B_1})\mathcal H^2\res\partial B_1,
$$
where $\nu$ is the radial vector. On the other hand, by the zero flux condition on $Y$,
$$
\partial T_2\res \partial B_{1+\epsilon}(1)=0= \partial T_2\res \partial B_1(1),
$$ 
and by homological reasons this implies that the two boundary parts above are themselves boundaries.
\begin{figure}[htp]
\centering
\scalebox{0.5}{\input{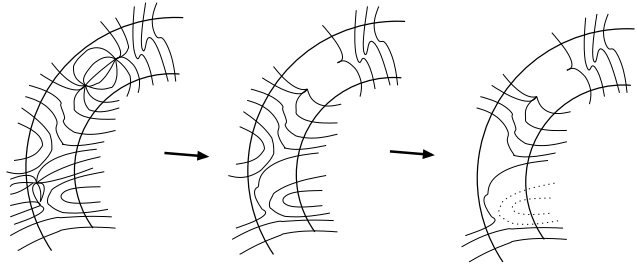}}
\caption{We represent schematically, to the left the decomposition of (the current associated to) the vector field $Y$ near $B_{1+\epsilon}\setminus B_1$, in the center the similar decomposition for $Y_1$, and to the right the vector field $Y_2$, where the part of the decomposition that will stay unmodified (and does not contribute to $Y_2$) is dotted. The result of subtracting $Y_2$ and adding $Y_3$ to $Y_1$ can be rephrased in a more picturesque way by saying that we are ``cancelling'' $Y_2$ and ``replacing it'' by $Y_3$. We eventually loose a bit in our estimates, since $Y_3$ ``forgets about the support'' of $Y_2$, and no easy form of a superposition principle holds for our range of exponents $p$.}\label{fig:weakstrong1}
\end{figure}
 So our strategy is to find another $L^p$-vector field $Y_3$ whose associated current $T_3$ has $\partial T_3=-\partial T_2\res \partial B_{1+\epsilon}$, and which has good norm estimates. The choice to which we are led is as follows:
$$
Y_3=\nabla f,
$$
for $f$ solving
\begin{equation}\label{neumann}
 \left\{\begin{array}{l}
        \Delta f=0\text{ on }B_{1+\epsilon}\setminus B_1,\\ 
        \partial_\nu f=g\text{ on }\partial B_1,\\
        \partial_\nu f=0\text{ on }\partial B_{1+\epsilon},
        \end{array}
\right.
\end{equation}
for $g:=-Y_2\cdot \nu$. Then we can define $\bar Y$ by extending $Y_3+Y_1-Y_2$ as zero outside $B_{1+\epsilon}$ and as $Y$ inside $B_1$.\\
The boundary of the associated current $T_{\bar Y}$ is equal to $(\partial T_Y)\res \op{int}B_1$, therefore $\bar Y\in \mathcal R_\infty(B_R)$.\\

 The only fact left to prove in order to obtain Proposition \ref{interpolant}, is the estimate of the $L^p$-energy of $\bar Y$, for which we need the following scaling lemma:
\begin{lemma}\label{scaling}
 There exists a constant $C$ depending only on the exponent $p$ but not on $\epsilon$, such that the following holds. For any $f\in W^{1,p}(B_{1+\epsilon}\setminus B_1)$ that is a weak solution of the Neumann boundary value equation \eqref{neumann}, where $g\in L^p(\partial B_1)$, the following estimate holds:
$$
||\nabla f||_{L^p(B_{1+\epsilon}\setminus B_1)}\leq C\epsilon^{-\frac{1}{p}}||g||_{L^p(\partial B_1)}.
$$
\end{lemma}
\begin{proof}
 We denote by $f_\epsilon$ a solution of \eqref{neumann} with parameter $\epsilon$. We observe that the weak formulation of the above Neumann problem states that for all $\phi\in C^\infty(\bar B_{1+\epsilon}\setminus B_1)$, 
$$
\int_{B_{1+\epsilon}\setminus B_1}\nabla f_\epsilon\cdot\nabla \phi =\int_{\partial B_1}\phi g d\mathcal H^2.
$$
Therefore, for any $\epsilon>0$ and for any test function $\phi$ on $\mathbb R^3$ there holds
\begin{equation}\label{neumconseq}
\int_{B_{1+\epsilon}\setminus B_1}\nabla f_\epsilon\cdot\nabla \phi =\int_{B_2\setminus B_1}\nabla f_1\cdot\nabla\phi .
\end{equation}
Now observe that the gradients form a closed subspace of $L^p(B_{1+\epsilon}\setminus B_1,\mathbb R^3)$, thus the following equality holds, where $q=\frac{p}{p-1}$:
\begin{equation}\label{eqnorm}
||\nabla f_\epsilon||_{L^p(B_{1+\epsilon}\setminus B_1)}=\sup\left\{\int_{B_{1+\epsilon}\setminus B_1}\nabla f_\epsilon\cdot\nabla \phi_\epsilon :\:||\nabla \phi_\epsilon||_{L^q(B_{1+\epsilon}\setminus B_1)}\leq 1\right\}.
\end{equation}
Here it is enough to consider functions $\phi_\epsilon$ belonging to $C^\infty(\bar B_{1+\epsilon}\setminus B_1)$. For any such test function, there exists another test function $\phi$ defined via the relation
$$
\phi_\epsilon(1+r,\theta)=\phi(1+r/\epsilon,\theta),\quad \forall r\in[0,\epsilon],\forall\theta\in S^2.
$$
The map $\phi\mapsto\phi_\epsilon$ is bijective between $C^\infty(\bar B_2\setminus B_1)$ and $C^\infty(\bar B_{1+\epsilon}\setminus B_1)$ and for a geometric constant $C_g\leq2$, there holds 
$$
||\nabla \phi_\epsilon||_{L^q(B_{1+\epsilon}\setminus B_1)}\leq C_g\epsilon^{\frac{1}{q}-1}||\nabla\phi||_{L^q(B_2\setminus B_1)}.
$$
This last fact and \eqref{neumconseq} can be applied to the equivalent definition \eqref{eqnorm}, immediately yielding our thesis. Indeed, we can obtain a constant $C$ as in the theorem's formulation, which depends on $C_g$ and on the constant of the classical $L^p$-regularity estimate for the Neumann problem on the domain $B_2\setminus B_1$, neither of which depends on $\epsilon$.
\end{proof}

We thus obtained an estimate of $||Y_3||_{L^p(B_{1+\epsilon}\setminus B_1)}$ via $\epsilon^{-\frac{1}{p}}||Y_2||_{L^p(\partial B_1)}$, and this suffices because of \eqref{esty2}. Moreover, $||Y_1-Y_2||_{L^p(B_{1+\epsilon}\setminus B_1)}\leq ||Y||_{L^p(B_{1+\epsilon}\setminus B_1)}$, because $(Y_1-Y_2)|_{B_{1+\epsilon}\setminus B_1}\preceq Y|_{B_{1+\epsilon}\setminus B_1}$. This concludes the proof of Proposition \ref{interpolant}.
\end{proof}

\subsection{Proof of Theorem \ref{weakstrong}}
In the proof of Theorem \ref{weakstrong} it is enough to consider the case $r=1$, and suppose $B=B_R, R>1$, since the general case follows via the scaling of the energy.\\

If the $X_k$ and the $X$ would be in $L^p\cap \mathcal R_\infty(B)$, then we would apply Proposition \ref{interpolant} to $Y_k=X_k-X$ on the shell $B_{1+\epsilon_0}\setminus B_\epsilon$. In general we cannot rely on this hypothesis, so we use the fact that $\mathcal R_\infty(B)$ is dense in $L^p_{\mathbb Z}(B)$ and complicate a bit our constructions.\\

\begin{proof}[Proof of Theorem \ref{weakstrong}:] We proceed in $3$ steps.\\
\textbf{Step 1: finding a spherical shell of small norm.} $X\in L^p(B)$ and the $X_k$ converge weakly to it, so by weak lowersemincontinuity, up to forgetting the first terms of the sequence $X_k$, there holds
$$
||X_k-X||_{L^p}\leq ||X_k||_{L^p}+||X||_{L^p}\leq 3||X||_{L^p}.
$$
We fix $\epsilon_0$ and we divide the interval $[1,1+\epsilon_0]$ in $M$ smaller intervals $I_h$ of length at least $\epsilon=\epsilon_0/2M$. Then, with the notation
$$
A_{I_h}=\{x\in\mathbb R^3:\:|x|\in I_h\},
$$
by pigeonhole principle we can find a subsequence of the $X_k$ and an index $h$ such that 
$$
||X_k-X||_{L^p(A_{I_h})}^p\leq \frac{C}{M}=\epsilon\frac{2C}{\epsilon_0}.
$$
From now on we forget about $h$, and call $I:=I_h$. Given any $\delta>0$, up to choosing another subsequence and changing $I$ slightly, we can also assume 
$$
||X_k-X||_{L^p(\partial B_{\inf I})}^p\leq \delta.
$$
\textbf{Step 2: approximating the interpolant.} At this point, with the notation $Y_k:=X_k-X$, we use the strong density of $\mathcal R_\infty(B)$ in $L^p_{\mathbb Z}(B)$ to find an approximant $\tilde Y_k\in \mathcal R_\infty(B)$ such that the $L^p$-distance of $Y_k$ and of $\tilde Y_k$ on $A_I$, as well as the $L^p$-distance of their boundary values, are not larger than $\epsilon_1$. Similarly we can define approximants $\tilde X_k, \tilde X$.\\
Up to changing $I$ slightly, we can insure that none of the $\tilde Y_k$ have any charges on $\partial A_I$, so that we can apply Proposition \ref{interpolant} to them. We obtain $\bar Y_k\in\mathcal R_\infty(B)$ that is 
\begin{itemize}
 \item $L^p$-close to $Y_k$ on $B_1$,
 \item zero om $B\setminus B_{1+\epsilon_0}$.
\end{itemize}
Up to passing to a subsequence there holds:
$$
\left\{\begin{array}{l}
        \tilde X_k-\bar Y_k\rightharpoonup \bar X_k\in L^p_{\mathbb Z}(B),\\
        \left.(\tilde X_k - \bar Y_k)\right|_{B_{1+\epsilon_0}\setminus B_1}\to X_k|_{B_{1+\epsilon_0}\setminus B_1},\\
        \left.(\tilde X_k - \bar Y_k)\right|_{B_1}\to X|_{B_1}.
       \end{array}
\right.
$$
The $\bar X_k$ defined as above (which depends of the choices of subsequences, on $I$, and on the parameters $\epsilon _1,\epsilon_0,\epsilon, \delta$), will be our choice of an interpolant between $X$ and $X_k$.\\

\textbf{Step 3: final norm estimates.} We can now patch together all our constructions and estimates to obtain the following chain of inequalities. We simplify the notations and write directly $||\cdot||_{X}$ instead of $||\cdot||_{L^p(X)}^p$.
\begin{eqnarray*}
 ||X_k||_{B_1}&\leq& ||X_k||_{B_{1+\epsilon_0}}\leq ||\bar X_k||_{B_{1+\epsilon_0}}\quad\text{by minimality of }X_k\\
 &\leq& ||X||_{B_1}+\liminf_{\epsilon_1\to0}||\bar Y_k||_{B_{1+\epsilon_0}\setminus B_1}\quad\text{by lowersemicontinuity}\\
&\leq&||X||_{B_1}+C\liminf_{\epsilon_1\to0}\left(||Y_k||_{A_I} +\epsilon_1+ C\epsilon^{-1}||Y_k||_{\partial B_{\inf I}} +\epsilon_1\right)\quad\text{ using Prop. }\ref{interpolant}\\
&\leq&||X||_{B_1} + C\left(\frac{\epsilon}{\epsilon_0} +\frac{\delta}{\epsilon_0}\right),
\end{eqnarray*}
and since there is no obstruction to letting $\epsilon,\delta$ be arbitrarily small, the desired inequality
$$
||X_k||_{L^p(B_1)}\leq||X||_{L^p(B_1)},
$$
holds and the thesis follows.
\end{proof}

\section{The regularity result}\label{secfinal}
\subsection{Dimension of the singular set}
\begin{definition}
For a vector field $X\in L^p(\Omega)$ defined on some domain $\Omega$, we define the \emph{regular set of $X$}, $\op{reg} (X)\subset\Omega$, as the set of those points in a neighborhood of which $X$ is $C^1$-regular. The set $\Omega\setminus \op{reg}(X):=\op{sing}(X)$ is called the \emph{singular set of $X$}.
\end{definition}

\begin{proposition}\label{dimsing}
If $X\in\mathbb L^p_{\mathbb Z}(\Omega)$ is a minimizer of the $L^p$-energy, then for $\Omega'\Subset\Omega$, $\mathcal H^{3-2p}(\op{sing}(X)\cap\Omega')=0$ and $\op{sing}(X)$ is nowhere dense in $\Omega'$.
\end{proposition}

\begin{proof}
Without loss of generality we suppose that $X$ is minimizing with respect to perturbations supported in a neighborhood $N$ of $\Omega$, and we prove the result with $\Omega$ instead of $\Omega'$. From Proposition \ref{ereg} we know that $x_0\in\op{reg}(X)$ if for some $r>0$ there holds
$$
r^{2p-3}\int_{B(x_0,r)}|X|^p\leq \epsilon_0.
$$
We can then cover $\op{sing}(X)$ by $2\delta $-balls $B_1^{2\delta},\ldots,B_l^{2\delta}$ contained in $N$ such that the balls $B_k^\delta$, having the same centers and radius $\delta$, are disjoint. Now, by monotonicity we obtain
$$
\epsilon_0\leq\delta^{2p-3}\int_{B_k^\delta}|X|^p,\quad k=1,\ldots,l.
$$
and summing this on $k$ we obtain
\begin{equation}\label{eqhausdorff}
l\delta^{3-2p}\leq \frac{1}{\epsilon_0}\int_{\cup B_k^\delta}|X|^p\leq \frac{||X||_{L^p(\Omega)}^p}{\epsilon_0}.
\end{equation}
After choosing such a family of balls for all $\delta$ we obtain the volume estimate 
$$
\mathcal H^3\left(\bigcup B_k^\delta\right)=l\delta^3\leq C\delta^2p\to 0,
$$
therefore by dominated convergence, 
$$
\int_{\cup B_k^\delta}|X|^p\to 0\text{ as }\delta\to 0.
$$
Inserting this in \eqref{eqhausdorff} gives, by definition of $\mathcal H^{3-2p}$ and by the covering property of our chosen balls, $\mathcal H^{3-2p}(\op{sing}(X))=0$, as wanted.\\
If we choose a ball $B\subset\Omega$ and we pack it as above with families $\mathcal F_\delta$ of small disjoint balls of radiuses $\delta\to 0$, we see by the scaling reasoning as above that if $X$ has rescaled energy bounded from below by $\epsilon_0$ on all balls for all $\delta$, then $X$ has to have infinite energy on $B$, which is not the case. Therefore there is a small ball on which the $\epsilon$-regularity theorem \ref{ereg} holds, showing that $\op{sing}(X)$ is nowhere dense.
\end{proof}

\subsection{Singular set of weak limits of minimzing vector fields}
\begin{proposition}\label{singseqmin}
 Suppose that $X_k$ are minimizers and $X_k\rightharpoonup X_0$. Then $X_k\to X_0$ locally uniformly on $\Omega'\setminus S_0$, for any $\Omega'\Subset \Omega$. Moreover $S_0$ is contained in the energy concentration set  
$$
\Sigma:=\left\{x\in\Omega:\:\liminf_{k\to\infty}\lim_{r\to 0}r^{2p-3}\int_{B(x,r)}|X_k|^p>4\epsilon_0\right\},
$$
where $\epsilon_0$ is the constant of the $\epsilon$-regularity theorem \ref{ereg}, and $\mathcal H^{3-2p}(\Sigma\cap\Omega')=0$.
\end{proposition}
\begin{rmk}
 It can be shown that $S_0=\Sigma$, but we don't need this characterization.
\end{rmk}
\begin{proof}
 We can assume up to taking a subsequence that $X_k\to X_0$ strongly in $L^p$. We show that $\mathcal H^{3-2p}_\infty(\Sigma)=0$, and that outside $\Sigma$ the $X_k$ converge uniformly; this is equivalent to the thesis.\\

 It follows, directly from its definition, that $\Sigma$ can be covered by finitely many balls $B_i$, with centers in $\Sigma$ and radiuses $r_i$, and such that for $k$ large enough, 
$$
(2r_i)^{2p-3}\int_{2B_i}|X_k|^p>2\epsilon_0,\quad\text{for all }k,i.
$$
We fix the choice of this set of balls, such that 
$$
\sum_i r_i^{3-2p}\leq \mathcal H_\infty^{3-2p}(\Sigma)+\epsilon.
$$
Then, by the estimates of the $\epsilon$-regularity, it follows that $X_k$ are uniformly H\"older on $\Omega'\setminus\bigcup B_i$, and therefore they have a subsequence converging uniformly on that set. By the reasoning of the proof of Proposition \ref{dimsing}, as $\delta\to 0$ the sum $\sum r_i^{3-2p}$ must converge to zero, and by the arbitrarity of $\epsilon$ above it follows that $H_\infty^{3-2p}(\Sigma)=0$. \end{proof}

\begin{corollary}\label{schssing}
 Let $X_k$ be a minimizer of the $L^p$-energy, $X_k\rightharpoonup X_0$ and $S_k:=\op{sing}(X_k)$ for $i\geq 0$, and $s\geq 0$. Then for any $\Omega'\Subset\Omega$ there holds
$$
\mathcal H_\infty^s(S\cap \Omega')\geq\limsup_{k\to\infty}\mathcal H_\infty^s(S_k\cap \Omega').
$$
\end{corollary}
\begin{proof}
 Consider the balls $B_k$ as in Proposition \ref{singseqmin}, except that this time they are used to approximate $\mathcal H^s_\infty(S)$. Then for $k$ large enough there holds
$$
S_k\subset\bigcup B_i,
$$
and therefore we can obtain
$$
\mathcal H_\infty^s(S\cap \Omega')+\epsilon\geq\lim_{k\to\infty}\mathcal H_\infty^s(S_k\cap \Omega'),
$$
as wanted.
\end{proof}

\subsection{Monotonicity and tangent maps}\label{tangmaps}
We consider now a sequence of blow-ups of a minimzer $X$ around a point $x_0$. We call $X_r(x)=\frac{1}{r^2}X(rx+x_0)$, and we observe that 
$$
X\in L^p_{\mathbb Z}(B_r(x_0))\Leftrightarrow X_r\in L^p_{\mathbb Z}(B)
$$
\begin{proposition}\label{monotonicity}[Monotonicity formula]
If $X\in L_{\mathbb Z}^p$ is a minimizer of the $L^p$-energy, then for all $x\in B$ and for almost all $r<\op{dist}(x,\partial B)$ there holds
\begin{equation}\label{eq:monotonicity}
 \frac{d}{dr}\left(r^{2p-3}\int_{B_r(x)}|X|^pd\mathcal H^3\right)=2p\;r^{2p-3}\int_{\partial B_r(x)}|X|^{p-2}|X^\parallel|^2d\mathcal H^2
\end{equation}
where $X^\parallel$ is the component of $X$ orthogonal to $\partial/\partial r$.
\end{proposition}
%
Since the right hand side is positive, the left hand side has a limit $L(x)$ for $r\to 0^+$, so we can integrate equation \eqref{eq:monotonicity} from $0$ to $\lambda$, getting
$$
\lambda^{2p-3}\int_{B_\lambda}|X|^pd\mathcal H^3 - L=2p\int_{B_\lambda}r^{2p-3}|X|^{p-2}|X^\parallel|^2 d\mathcal H^3.
$$
As in \cite{HL}, the function $L(x)$ is actually upper semi-continuous.\\

The equation \eqref{eq:monotonicity} also implies that 
$$
\int_{B_1}|X_r|^p=E_r(X):=r^{2p-3}\int_{B_r}|X|^p
$$ 
is increasing in $r$, therefore the $X_\lambda$ have a $L^p$-weakly convergent subsequence $X_{\lambda_i}\rightharpoonup X_0\in L^p$, $\lambda_i\to0$. By a change of variables in the integrated formula we obtain 
$$
\lambda^{2p-3}\int_{B_\lambda}|X|^pd\mathcal H^3 - L=2p\int_{B_1}r^{2p-3}|X_\lambda|^{p-2}|X_\lambda^\parallel|^2 d\mathcal H^3,
$$
therefore 
\begin{equation}\label{radialitylimit}
\lim_{\lambda\to 0^+}\int_{B_1}r^{2p-3}|X_\lambda|^{p-2}|X_\lambda^\parallel|^2 d\mathcal H^3=0,
\end{equation}
Since $p'=\tfrac{p}{p-1}$ and $X_0\in L^p$, we obtain that $|X_0|^{p-2}X_0\in L^{p'}$; the weight $r^{2p-3}$ actually worsens the convergence above since it's bounded away from zero, so we obtain that $X_0^\parallel=0$. This proves more in general the following:
\begin{proposition}\label{radiality}
For any minimizer $X$, for any $x\in \op{int}( B)$ and for any sequence of rescalings $X_{x,\lambda_i}$ around $x$, with $\lambda_i\to 0$, the weak accumulation points $X_{x,0}$ are radially directed.
\end{proposition}

\subsection{Stationarity and dimension reduction for the singular set}\label{statdimred}
From now on we call $s$ any exponent (smaller than $3-2p$, as seen above) for which $\mathcal H^s(S\cap \Omega')>0$, where $S=\op{sing} (X)$ for a minimizer $X$. Except for $x_0$ in a set $S'$ such that $\mathcal H^s(S')=0$, there holds
\begin{equation}\label{hslim}
\liminf_{\lambda\to 0}\lambda^{-s}\mathcal H^s(S\cap B_{\lambda/2})>0,
\end{equation}
where the balls $B_{\lambda/2}$ are all centered at $x_0$. As in the previous section, for a subsequence $\lambda_i\to 0$ our blow-ups converge to a radial tangent map $X_0$ weakly in $L^p$, and since they are all minimizers, by Theorem \ref{weakstrong} they converge strongly, up to taking another subsequence, and also $X_0$ is a minimizer.\\

The singular set $S_i$ of $X_{\lambda_i}$ is the blowup of $S$, and 
$$
\lambda_i^{-s}\mathcal H^s(S\cap B_{\lambda_i/2})=\mathcal H^s(S_i\cap B_{1/2})
$$
and from \eqref{hslim} we follow that
\begin{equation}\label{singmeas}
 \mathcal H^s(S_0\cap B_{1/2})>0,
\end{equation}
where $S_0$ is the singular set of $X_0$.\\

Using the radial direction of $X_0$ and the stationarity (Prop. \ref{stationarity}) we now show the following fact.
\begin{lemma}\label{x0invar}
 For any minimizer $X$, any tangent map $X_0$ satisfies
\begin{equation*}
 |X_0|(x)=|x|^{-2}|X_0|(x/|x|).
\end{equation*}
\end{lemma}
\begin{proof}
 We use the equation \eqref{stationarity} with respect to a local frame $e_1,e_2,e_3$ such that the vector $e^3$ is the radial one and $\omega$ associated to $X$ has just the component parallel to $de^1\wedge de^2$ different from zero (as was proved in Proposition \ref{radiality}), and we consider a perturbation field that can be expressed in polar coordinates $(\rho,\theta)$ as
$$
V(\rho,\theta)=f(\rho)\phi(\theta)\hat \rho.
$$
We then get from \eqref{stationarity} that
$$
 0=p\int|\omega|^p(\rho,\theta)\frac{1}{\rho}f(\rho)\rho^2d\rho\;\phi(\theta)d\theta - \int|\omega|^p(\rho,\theta)\frac{1}{\rho^2}\partial_\rho(\rho^2f(\rho))\rho^2d\rho\;\phi(\theta)d\theta,
$$
By the arbitrarity of $\phi(\theta)$ this translates into the following equation holding for almost all $\theta$
$$
 \int|\omega|^p(\rho,\theta)\left[2(p-1)\rho f(\rho) - \rho^2 f'(\rho)\right]d\rho=0.
$$
This can also be written in terms of $F(\rho)=\rho^{-2p}f(\rho)]$ as
$$
\int|\omega|^p(\rho,\theta)\rho^{2p}F'(\rho)d\rho=0
$$
and since this holds for all $F$ with support contained in $]0,\infty[$, it must be that 
$$
|\omega|(\rho,\theta)\rho^2\text{ is independent of }\rho,
$$
as wanted.
\end{proof}

Along the same lines as the above proof (we just have to redefine the the orthonormal frames properly), we obtain the following result without difficulty:
\begin{lemma}\label{dirinvar}
 If $X_1$ is parallel to one coordinate direction $e_3$ and if the stationarity equation holds, then $X_1$ is almost everywhere independent of the coordinate $x_3$. In particular the thesis is satisfied if $X_1$ minimizes the energy.
\end{lemma}

\begin{rmk}
 We note that in Sections \ref{tangmaps} and \ref{statdimred} until this point just the monotonicity and stationarity formulas were used, without the intervention of any comparison argument. Thus the results proved so far in this subsection are valid not only when $X$ is a minimizer, but also when $X$ is just stationary, i.e. the $2$-form $F$ associated to it satisfies 
$$
\left.\frac{d}{dt}\right|_{t=0}\int_{B^3}|\phi_t^*F|^p =0,
$$
for all families of diffeomorphisms $\phi_t:B^3\to B^3$ that are differentiably dependent on $t\in[-1,1]$, equal to the identity in a neighborhood of $\partial B^3$, and such that $\phi_0=id_{B^3}$. This requirement is indeed enough to prove stationarity and monotonicity. On the contrary, the dimension reduction technique that we are about to prove uses the strong convergence result which in turn depends on a comparison argument, thus the following proofs hold only for minimizers $X$. An intriguing open question is whether or not the uniqueness of tangent maps holds in our case (see Section 6.2 of \cite{P3} for a broader discussion).
\end{rmk}

We are now ready to apply the dimension reduction technique of Federer to our minimizing vector field $X$. We start with a radial tangent map $X_0$, obtained by blow-up at a point $x_0$ at which $S_0$ has positive density with respect to $\mathcal H^s$ for some $s<3-2p$ as above, as in \eqref{singmeas}.\\
As we saw in Section \ref{tangmaps}, $X_0$ is a strong limit of a blowup sequence relative to some $\lambda_i\to 0^+$. We also know that the singular set $S_0$ of $X_0$ has zero $\mathcal H^{3-2p}$-measure and is nowhere dense. It follows from Lemma \ref{x0invar}, that $|X_0|$ must be $(-2)$-homogeneous, and $\op{div}X_0=0$ locally outside $S_0$. Therefore $X_0$ is itself $(-2)$-homogeneous outside $S_0$, and $S_0$ is radially invariant, i.e. 
$$
\lambda S_0\subset S_0,\quad\forall \lambda>0.
$$
Now we show that $S_0=\{0\}$. Indeed, were this not the case, we could find a point $x_1\in S_0\cap B_{1/2}$. In this case we could blow up again $X_0$ with center $x_1$, obtaining a tangent map $X_1$. By strong convergence we obtain that $X_1$ would have to be both directed radially and directed along one fixed direction: this would imply that $X_1=0$, contradicting the fact that $x_1\in S_0$.\\

The following proposition summarizes the above discussion.
\begin{proposition}\label{dimzero}
 For a minimizing vector field $X$, the singular set of any tangent map $\op{sing}(X_0)$ is either empty or contains just the origin.
\end{proposition}

After Proposition \ref{dimzero} we deduce our main result easily.

\begin{theorem}\label{mainthm2}
 A minimizer $X$ must have finitely many isolated singularities in any open $\Omega'\Subset\Omega$.
\end{theorem}
\begin{proof}
 If $X$ had an accumulating sequence of singular points $\op{sing}(X)\ni x_i\to x\in\Omega'$, then we can select a small $r>0$ such that $B(x,r)\subset\Omega'$. Then we can consider the distances
$$
\lambda_i=\frac{|x-x_i|}{4},
$$
and we observe that for the blowups $S_i$ of ratio $\lambda_i$ and center $x$, there holds $\mathcal H^0(S_i\cap B_{1/2})>2$. This contradicts Proposition \ref{dimzero} (where $\mathcal H^0(S_0)\leq1$) and the semicontinuity proved in Corollary \ref{schssing}.
\end{proof}

\section{Stationarity and monotonicity}\label{statmon}
\subsection{Stationarity formula}

We consider a smooth diffeomorphism $\varphi_t:=id+tV$, where $V$ is a compactly supported vector field and $t$ is small enough. We compute the stationarity formula arising from
$$
\left.\frac{d}{dt}\int |\varphi_t^*\omega|^p\right|_{t=0}=0.
$$
We recall the formula of the norm of the pullback of $\omega$ via $\varphi_t$, with respect to an orthonormal frame field $e_1,e_2,e_3$:
\begin{eqnarray*}
 |(\varphi_t^*\omega)_x|^2&=&\sum_{i,j=1}^n\left|\omega_{\varphi_t(x)}(d\varphi_te_i,d\varphi_te_j)\right|^2\\
&=&\sum_{i,j=1}^n\left|\omega_{\varphi_t(x)}(e_i + t\;dV\cdot e_i,e_j+t\;dV\cdot e_j)\right|^2.
\end{eqnarray*}
To deal better with the above $t$-derivative, we change variable (we let $y:=\varphi_t^{-1}(x)$), so that the point at which we calculate the norm of $\omega$ does not depend on $t$:
$$
\int |\varphi_t^*\omega_x|^pdx=\int\left[\left(\sum_{i,j=1}^n|\omega_y(e_i + t\;dV\cdot e_i,e_j+t\;dV\cdot e_j)|^2\right)^{p/2}\det(id+tdV)^{-1}\right]dy
$$
Now we take the derivative of the integrand in $t=0$, obtaining by easy computations (see for example \cite{Price}):
\begin{equation}\label{stationarity}
 p\int|\omega|^{p-2}\sum_{i,j=1}^3\omega(e_i,e_j)\omega(\nabla_{e_i}V,e_j) - \int|\omega|^p\op{div}V=0.
\end{equation}
The above formula is justified for minimization problems in $L^p$, because we are sure that the manipulations done extend to that setting. What ensures that doing the pullback preserves the property of being in $L^p_{\mathbb Z}$ as well, is the following:
\begin{proposition}
Consider a regular foliation
$$
\{\Sigma_\lambda^2: \; \lambda\in[-\epsilon,\epsilon]\},
$$
i.e. a parameterized set of $2$-surfaces in $\mathbb R^3$ such that if $N_\epsilon\Sigma=\cup_\lambda\Sigma_\lambda^2$, then the following (has sense and) holds:
$$
\int_{N_\epsilon\Sigma}X\cdot\nu_{\Sigma_\lambda^2}d\mathcal H^3\simeq \int_{-\epsilon}^\epsilon\int_{\Sigma_\lambda^2}X\cdot\nu_{\Sigma_\lambda^2}d\mathcal H^2\;d\lambda,
$$
where $\nu_{\Sigma_\lambda^2}$ is the normal vector of $\Sigma_\lambda^2$.\\

The following property is equivalent to the fact that $X\in L^p_{\mathbb Z}$:\\
For almost all $\lambda\in [-\epsilon, \epsilon]$ the following holds:
\begin{equation}\label{integer}
 \int_{\Sigma_{\lambda}^2}X\cdot\nu\;d\mathcal H^2\in\mathbb Z.
\end{equation}
\end{proposition}

\begin{proof}
 This follows since $\mathcal R^\infty$ is dense in $L^p_{\mathbb Z}$ in the $L^p$-norm. Suppose indeed that there exists a closed $C^2$-surface $\Sigma$ such that for a set of $\lambda\in[-\epsilon,\epsilon]$ of measure $\delta>0$ there holds
$$
\int_{\Sigma_\lambda^2} X\cdot\nu d\mathcal H^2\in]a+c,a+1-c[,\quad\text{for some }a\in\mathbb Z.
$$
In particular, whenever $X_i\stackrel{L^p}{\to}X$, $X_i\in R^\infty_\varphi$ then
\begin{eqnarray*}
 \int_{N_\epsilon\Sigma}|X_i - X|^pd\mathcal H^3 &\geq& C\int_{-\epsilon}^\epsilon\int_{\Sigma_\lambda^2}|X_i - X|^pd\mathcal H^2\;d\lambda\\
&\geq&C\int_{-\epsilon}^\epsilon|\Sigma_\lambda^2|^{1-p}\left|\int_{\Sigma_\lambda^2}X_i\cdot\nu d\mathcal H^2 - \int_{\Sigma_\lambda^2}X\cdot\nu d\mathcal H^2\right|^pd\lambda\\
&\geq&C\delta c^p,
\end{eqnarray*}
contradicting the convergence in $L^p$-norm stated above.
\end{proof}

Since for $t<||X||_\infty/2$, it follows that $\phi_t$ is a diffeomorphism and the integral of $\omega$ on a sphere $S$ is by definition the same as the integral of $\phi_t^*\omega$ on $\phi_t^{-1}(S)$, we see by the above proposition that $\omega\in\mathcal F_{\mathbb Z}^p(\Omega)$ implies that also the perturbations $\phi_t^*\omega$ belong to the same space for $t$ small.

\subsection{Monotonicity formula}
In this section we prove a refinement of the stationarity formula. Since the proof is independent if the dimension $n$ of our domain, we give a formulation in any dimension (the defnition of $\mathcal F_{\mathbb Z}^p(\Omega)$ now requiring the degree to be an integer on any $2$-dimensional sphere). For our applications we will just use the case $n=3$.
\begin{proposition}[Monotonicity formula]
 If $\omega\in\mathcal F_{\mathbb Z}^p$ is stationary, then for all $x$ and almost all $r\in]0,R]$ with the constraint $B_R(x)\subset\Omega$ there holds
\begin{equation}\label{eq:monotonicityend}
 \frac{d}{dr}\left(r^{2p-n}\int_{B_r}|\omega|^pdy\right)=2p\;r^{2p-n}\int_{\partial B_r}|\omega|^{p-2}|\partial_\rho\lrcorner\omega|^2d\sigma
\end{equation}
where $\partial_\rho=\frac{\partial}{\partial\rho}$ is the radial derivative.
\end{proposition}
\begin{proof}

We use a strategy similar to \cite{HL}. If $F:B_R\to B_R$ is a weakly differentiable bijective Lipschitz function, and if $\omega\in\mathcal F_{\mathbb Z}^p$ then also $F^*\omega\in\mathcal F_{\mathbb Z}^p$, so it is a competitor in our minimization. Therefore the stationarity
$\left.\frac{d}{dt}\int_{B_R}|F_t^*\omega|^p\right|_{t=1}=0,$ holds provided that $F_0=id_{B_R}$ and that the family $F_t$ is differentiable in $t$. Such properties will be clear from our choices of the map $F$.
\eqref{eq:monotonicityend} follows from this. \\

\textbf{Definition of $F$}\\
Fix $0<r<s<R$ such that $0<t<s/r$. Then we define a function $F=F_{r,s,t}:B_R\to B_R$ by $F(x):=\eta(|x|)x$, such that $$
\rho:=|x|\mapsto |F(x)|
$$
is continuous and affine on each of the intervals $[0,r],[r,s],[s,R]$. We define
\begin{equation}\label{eq:defF}
\eta(\rho)=\left\{\begin{array}{lll}
             t&\text{ if }\rho\leq r\\
             1&\text{ if }\rho\in [s,r]
            \end{array}
\right.
\end{equation}
and $\eta|_{]r,s[}$ is defined accordingly:
$$
\eta|_{[r,s]}(\rho):=\frac{s-tr}{s-r} + \frac{1}{\rho}\frac{rs(t-1)}{s-r}.
$$

\textbf{Expression of $|F^*\omega|^2$}\\
We do our computation in coordinates. We choose a basis $\{e_0,e_1,\ldots,e_{n-1}\}$ with respect to which to write the matrix $dF_x$, where $e_0=\partial_\rho$ and the other vectors form an orthogonal basis together with it. Then
\begin{equation}\label{eq:dF}
 \frac{\partial F}{\partial x_k}=\eta e_k +\rho\eta'\delta_{0k} e_0.
\end{equation}

Then 
\begin{eqnarray*}
|(F^*\omega)_x|^2&=&\sum_{i,j}\left[\omega_{F(x)}(dF_xe_i,dF_xe_j)\right]^2\\
&=&\sum_{i,j=0}^{n-1}\left|\omega\left(\frac{\partial F}{\partial x_i},\frac{\partial F}{\partial x_j}\right)\right|^2\\
&=&\sum_{i,j>0}|\omega(\eta e_i, \eta e_j)|^2 + 2\sum_{i>0}|\omega((\eta +\rho\eta')e_0,\eta e_i)|^2\\
&=&\eta^4\sum_{i,j>0}\omega_{ij}^2 + 2\eta^2(\eta +\rho\eta')^2|\partial_\rho\lrcorner\omega|^2.
\end{eqnarray*}

\textbf{The derivative in $t$}\\
We now start the computations for the monotonicity formula. 
\begin{equation}\label{eq:1}
 \int_{B_R}\left|F^*\omega\right|^p=I+II+III
\end{equation}
where, after a change of variables $y=F^{-1}(x)$,
\begin{eqnarray*}
I:&=&\int_{B_r}\left|F^*\omega\right|^p=t^{2p-n}\int_{B_{rt}}|\omega|^pdy,\\
II:&=&\int_{B_s\setminus B_r}\left|F^*\omega\right|^p\\
III:&=&\int_{B_R\setminus B_s}\left|F^*\omega\right|^p=\int_{B_R\setminus B_s}\left|\omega\right|^pdy
\end{eqnarray*}
We want now to change variable also in $II$ and to take $\left.\frac{d}{dt}\right|_{t=1}$ of the terms above. The easy terms give:
\begin{eqnarray*}
I':=\left.\tfrac{d}{dt}\right|_{t=1}(I)&=&(2p-n)\int_{B_r}|\omega|^pdy+r\int_{\partial B_r}|\omega|^pd\sigma\\
III':=\left.\tfrac{d}{dt}\right|_{t=1}(III)&=&0
\end{eqnarray*}

\textbf{Ingredients for the computations}\\
\begin{itemize}
 \item We observe that 
$$
\eta(\rho) +\rho\eta'(\rho)=\frac{s-tr}{s-r},
$$
which has $t$-derivative $\frac{-r}{s-r}$. It is useful to keep in mind that $\eta=1$ for $t=1$; this will be used without mention in the calculations.
\item If $y=F(x)$ and $\sigma:=|y|$, then we can write the expression of $\eta$ in terms of $\sigma$:
$$
\sigma=\rho\eta(\rho)=(\rho-r)\frac{s-tr}{s-r} +tr
$$
so
$$
\rho=f(\sigma):=(s-r)\frac{\sigma-tr}{s-tr} +r
$$
and
$$
\eta(f(\sigma))=\frac{s-tr}{s-r} +\left[(s-r)\frac{\sigma - tr}{s-tr}+r\right]^{-1}\frac{rs(t-1)}{s-r},
$$
whence
$$
\left.\frac{d}{dt}\eta(f(\sigma))\right|_{t=1}=-\frac{r}{s-r} +\frac{rs}{\sigma(s-r)}.
$$
\item From \eqref{eq:dF} it follows that for $\rho:=|x|\in[r,s]$,
$$
 J(dF^{-1})=\left[\eta(\rho)^{n-1}(\eta(\rho)+\rho\eta'(\rho))\right]^{-1}, 
$$
so
\begin{eqnarray*}
 \left.\frac{d}{dt}J(dF^{-1})\right|_{t=1}&=&(1-n)\left.\frac{d}{dt}\eta\right|_{t=1} + \left.\frac{d}{dt}\left(\frac{s-r}{s-tr}\right)\right|_{t=1}\\
&=&(1-n)\left[-\frac{r}{s-r} + \frac{rs}{\sigma(s-r)}\right] +\frac{r}{s-r}.
\end{eqnarray*}
\end{itemize}

\textbf{The computation of the $t$-derivative}\\
We call $|i^*\omega_y|^2 =\sum_{i,j>0}\omega_{ij}^2$ and we obtain
\begin{eqnarray*}
 II&=&\int_{B_s\setminus B_{rt}}\left(|i^*\omega_y|^2\eta^4 + 2\left(\frac{s-tr}{s-r}\right)^2 \eta^2 |\omega_y(\hat y,\cdot)|^2 \right)^{p/2} J(dF^{-1})dy\\
II'&:=&\left.\frac{d}{dt}II\right|_{t=1}\\
\text{(derivative of the domain)}&=&-r\int_{\partial B_r}|\omega_y|^p d\sigma\\
\quad\text{(der. of the Jacobian)}&&+(n-1)\frac{r}{s-r}\int_{B_s\setminus B_r}\left(1-\frac{s}{|y|}\right)|\omega_y|^p dy + \frac{r}{s-r}\int_{B_s\setminus B_r}|\omega_y|^p dy\\
\quad\text{(der. of the main term)}&&\frac{p}{2}\left[\frac{r}{s-r}\int_{B_s\setminus B_r}4|\omega|^p\left(\frac{s}{|y|}-1\right)dy - \frac{r}{s-r}\int_{B_s\setminus B_r}4|\omega|^{p-2}|\partial_\rho\lrcorner\omega|^2dy\right].
\end{eqnarray*}
\textbf{We now take the limit $s\downarrow r$} and we are interested in seeing what the equation $I'+II'+III'=0$ becomes. The answer is
\begin{eqnarray*}
\lim_{s\downarrow r}II'&=&-r\int_{\partial B_r}|\omega|^p d\sigma\\
&&+0+r\int_{\partial B_r}|\omega|^pdy\\
&&+0-2pr\int_{\partial B_r}|\omega|^{p-2}|\partial_\rho \lrcorner\omega|^2dy\\
&=&-2pr\int_{\partial B_r}|\omega|^{p-2}|\partial_\rho \lrcorner\omega|^2dy,
\end{eqnarray*}
and 
$$
 \lim_{s\downarrow r}I'=(2p-n)\int_{B_r}|\omega|^pdy + r\int_{\partial B_r}|\omega|^pdy
$$

Summing up and using the fact that $\omega$ is a minimizer of the energy, we get
$$
(2p-n)\int_{B_r}|\omega|^pdy+r\int_{\partial B_r}|\omega|^pdy=2pr\int_{\partial B_r}|\omega|^{p-2}|\partial_\rho \lrcorner\omega|^2dy
$$
Multiplying both the r.h.s. and the l.h.s of the above equation by $r^{2p-n-1}$ we get the desired formula
$$
\frac{d}{dr}\left(r^{2p-n}\int_{B_r}|\omega|^pdy\right)=2p\;r^{2p-n}\int_{\partial B_r}|\omega|^{p-2}|\partial_\rho \lrcorner\omega|^2dy
$$
\end{proof}

In terms of vector fields, we can state the following:

\begin{proposition}[Monotonicity formula, alternative formulation]
 If $X\in L_{\mathbb Z}^p$ minimzes the energy, then for almost all $r\in[0,R]$ there holds
\begin{equation}\label{eq:monotonicity2}
 \frac{d}{dr}\left(r^{2p-n}\int_{B_r}|X|^pdy\right)=2p\;r^{2p-n}\int_{\partial B_r}|X|^{p-2}|X-\langle X,\nu_{B_r}\rangle\nu_{B_r}|^2d\mathcal H^2
\end{equation}
\end{proposition}

\bibliographystyle{amsalpha}

 \end{document}